\documentclass[11pt]{article}
\usepackage{bcc23}
\usepackage[T1]{fontenc}\usepackage{graphicx}\usepackage{color}

\bcctitle{Counting  planar maps, coloured or uncoloured}
\bccshorttitle{Counting  planar maps, coloured or uncoloured}

\bccname{Mireille Bousquet-M\'elou}
\bccshortname{Mireille Bousquet-M\'elou}

\bccaddressa{CNRS, LaBRI, Universit\'e Bordeaux 1\\ 351 cours de la
  Lib\'eration\\ F-33405 Talence Cedex\\ France}
\bccemaila{bousquet@labri.fr}

\graphicspath{{Figures/}}

\newcommand{\Q}{{\bf Q}}

\newcommand{\be}{\beta}

\newcommand{\GK}{\mathbb{K}}

\newcommand{\beq}{\begin{equation}}
\newcommand{\eeq}{\end{equation}}

\newcommand{\gf}{generating function}
\newcommand{\gfs}{generating functions}
\newcommand{\fps}{formal power series}

\DeclareMathOperator{\TR}{TR}

\DeclareMathOperator{\df}{df}
\DeclareMathOperator{\dv}{dv}
\DeclareMathOperator{\cc}{c}
\DeclareMathOperator{\vv}{v}
\DeclareMathOperator{\ff}{f}
\DeclareMathOperator{\ee}{e}

\DeclareMathOperator{\Tpol}{T}
\DeclareMathOperator{\Ppol}{P}

\newcommand{\mM}{\mathcal{M}} 
\newcommand{\gM}{M} 
\newcommand{\MM}{M} 

\newcommand{\gtM}{\tilde{M}} 
\newcommand{\gtQ}{\tilde{Q}}  
\newcommand{\BM} {B} 

\newcommand{\gT}{T} 

\newcommand{\gB}{B} 
\newcommand{\gBT}{B^{ {\triangleleft}}} 
\newcommand{\ET}{E} 
\newcommand{\NT}{T} 
\newcommand{\mNT}{\mathcal{T}} 
\newcommand{\NQ}{Q} 

\def\emm#1,{{\em #1}}

\newcommand{\bx}{\bar x}
\newcommand{\by}{\bar y}

\newcommand{\bu}{\bar u}

\newcommand{\bv}{\bar v}
\newcommand{\bw}{\bar w}
\newcommand{\gQ}{Q}

\newcommand{\parfrac}[2]{\left( \frac{#1}{#2} \right)}

\begin{document}
\makebcctitle

\begin{abstract}
We present  recent results on the enumeration of $q$-coloured
planar maps, where each monochromatic edge carries a weight
$\nu$. This is equivalent to weighting each map by its Tutte
polynomial, or to solving the $q$-state Potts model on random planar maps. 
The associated \gf, obtained by Olivier Bernardi and the author, 
is differentially algebraic. That is, it satisfies a (non-linear) differential equation. The
starting point of this result is a functional equation written by Tutte in 1971,
which translates into enumerative terms a simple recursive description
of planar maps. The proof follows and adapts Tutte's solution of
\emm properly, $q$-coloured triangulations (1973-1984).

We put this work in perspective with the much better understood
enumeration of families of \emm uncoloured, planar maps, for which the recursive
approach almost
systematically yields algebraic \gfs.  In the past 15 years, these
algebraicity properties have been explained combinatorially by
illuminating bijections between maps and  families of plane trees. We survey both approaches, recursive and bijective.

Comparing the coloured and uncoloured results raises the question of 
designing bijections for coloured maps. No complete  bijective
solution exists at the moment, but we present bijections for
certain specialisations of the general problem. We also show that for
these specialisations, Tutte's functional equation is much easier
to solve that in the general case.

We conclude with some open questions.

\end{abstract}

\section{Introduction}

A planar map is a proper embedding in the sphere of a finite connected
graph, defined up to continuous deformation.  The enumeration of these objects has been a topic of constant
interest for 50 years, starting with a series of papers by Tutte  in the early
1960s; these papers were mostly based on recursive descriptions of
maps
(e.g.~\cite{tutte-triangulations}). The last 15
years have  witnessed a new burst of activity in this field, 
with the development of rich bijective approaches~\cite{Sch97,BDG-blocked}, and their
applications to the study of random maps of large size~\cite{le-gall-maps,marckert-miermont}. In such
enumerative problems, maps are usually \emm rooted, by orienting one edge.
Figure~\ref{fig:two-edges} sets a first exercise in map enumeration.

\begin{figure}[htb]
  \begin{center}
 \includegraphics[scale=0.4]{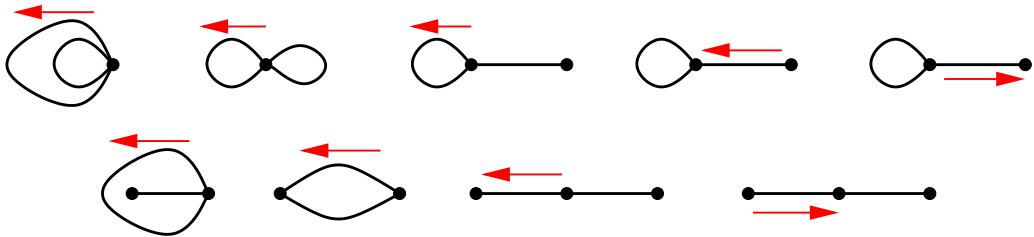}
  \caption{There are 9 rooted planar maps with two edges.}
\label{fig:two-edges}
  \end{center}
\end{figure}

Planar maps are not only studied in combinatorics and probability, but
also in theoretical physics.  In this context, maps are considered as
random surfaces, and constitute a model of 2-dimensional \emm quantum
gravity,. For many years, maps were studied independently in
combinatorics and in physics, and another approach for counting them, based on the
evaluation of certain matrix integrals, was introduced in the 1970s in
physics~\cite{BIPZ,BIZ}, and much developed since then~\cite{DFGZJ,mehta}. More recently, a fruitful
exchange started between the two communities. Some physicists
have become masters in combinatorial methods~\cite{BDG-planaires,bouttier-mobiles}, while the matrix
integral approach has been taken over by some probabilists~\cite{guionnet}.

From the physics point of view, it is natural to equip maps with 
additional structures, like particles, trees, spins,
and more generally classical models of statistical physics. In
combinatorics however, a huge majority of papers deal with the
enumeration of bare maps. There has been some exceptions to this rule
in the past few years, with combinatorial solutions of the Ising
and hard-particle models on planar maps~\cite{mbm-schaeffer-ising,BDG-hard-part-blossoming,BDG-blocked}. But there is also an earlier, and
major, exception to this rule: Tutte's  study of
properly $q$-coloured triangulations (Figure~\ref{fig:triang-coloree}).

\begin{figure}[ht!]\begin{center} \input{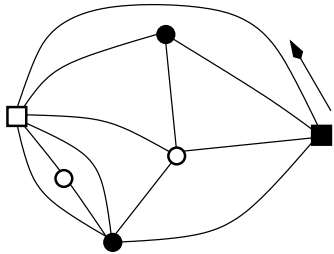}\caption{A (rooted) triangulation of the sphere, properly
  coloured with 4 colours.}\label{fig:triang-coloree} \end{center}\end{figure}

This ten years long study  (1973-1984) plays a
central role in this paper. For a very long time, it remained  an
isolated \emm tour de force, with no counterpart for other families of
planar maps or for more general colourings, probably because the
corresponding series of
papers~\cite{lambda12,lambda3,lambda-tau,tutteIV,tutteV,tutte-pair,tutte-chromatic-sols,tutte-chromatic-solsII,tutte-differential,tutte-chromatic-revisited}
looks quite formidable.  Our main point here is to report on recent advances in the
enumeration of (non-necessarily properly) $q$-coloured maps, in the
steps of Tutte. In the associated \gf, every monochromatic
edge is assigned a weight~$\nu$: the case $\nu=0$ thus captures proper
colourings. In physics terms, we are studying the $q$-state Potts
model on planar maps. A third equivalent formulation is that we count
planar maps weighted by their Tutte polynomial --- a 
bivariate generalisation of the chromatic polynomial, introduced by
Tutte, who called it the dichromatic polynomial. Since the Tutte
polynomial has numerous  interesting specialisations, giving for
example the number of trees, forests, acyclic orientations, proper colourings
of course,  or the partition function of the Ising model, or the reliability
and flow polynomials, we are covering several models at
the same time.

We shall put this work in perspective with the (much better
understood) enumeration of uncoloured maps, to which we devote
Sections~\ref{sec:uncoloured-rec} and~\ref{sec:bij-uncoloured}. We first present in
Section~\ref{sec:uncoloured-rec} the robust recursive approach found in the early work of
Tutte. It applies in a rather uniform way to many families of maps,
and yields for their \gfs\  functional equations that we call
\emm polynomial equations with one catalytic variable,. A typical
example is~\eqref{1cat-planaires1}. It is now
understood that the solutions of these equations are always algebraic,
that is, satisfy a polynomial equation. For instance, there are
$2\cdot 3^n {{2n}\choose n}/((n+1)(n+2))$ rooted planar
maps with $n$ edges, and their \gf, that is, the series
$$
M(t):=\sum _{n\ge 0}  \frac{2\cdot3^n}{(n+1)(n+2) }{{2n}\choose n} t^n,
$$
satisfies
$$
M(t)=1-16t+18tM(t)-27t^2M(t)^2.
$$

Thus algebraicity is
intimately connected with (uncoloured) planar maps. In Section~\ref{sec:bij-uncoloured},
we present two more recent bijective approaches that relate maps to
plane trees, which are algebraic objects \emm par excellence,.
Not only do these bijections give a
better understanding of algebraicity properties, but they also
explain why many families of maps are counted by simple formulas. 

In Section~\ref{sec:coloured-rec}, we discuss the recursive approach
for $q$-coloured 
maps. The corresponding functional equation~\eqref{eq:tM} was written in 1971
by Tutte ---who else?---, but was left untouched 
since then. It involves two ``catalytic'' variables, and it has been
known for a long time that its solution is not algebraic. The key point of this section, due to Olivier Bernardi and the author, is the solution of this
equation, in the form of a system of differential equations that
defines the \gf\ of $q$-coloured maps. This series is thus  \emm
differentially algebraic,, like Tutte's solution of properly coloured
triangulations. Halfway on the long path that leads to the solution
stands an interesting intermediate result: when $q\not = 4$ is of the
form $2 +2\cos (j\pi/m)$, for integers $j$ and $m$, the \gf\ of
$q$-coloured planar maps is algebraic. This includes the values
$q=2$ and $q=3$, for which we give explicit results.
 We also discuss certain
specialisations for which the equation becomes easier to solve, like
the enumeration of maps equipped with a bipolar orientation, or with a
spanning tree.

Since we are still in the early days of the enumeration of coloured
maps, it is not surprising that bijective approaches are at the moment
one step behind. Still, a few bijections are available for some of the
simpler specialisations mentioned above. They are presented in 
Section~\ref{sec:bij-coloured}. We conclude with open questions,
dealing with both uncoloured and coloured enumeration.

\medskip
This survey is sometimes written in an informal style, especially when
we describe bijections. Proofs are only
given when they are new, or especially simple and illuminating. The
reference list, although long, is certainly not exhaustive. In
particular, the papers cited in this introduction are just  examples
illustrating our topic, and should be considered as pointers to the
relevant literature. More references are given further
 in the paper. Two  approaches that have been used to count
maps are utterly absent from this paper: methods based on characters
of the symmetric group and symmetric functions~\cite{goulden-jackson,MR2442057}, which do not
exactly address the same range of problems, and the matrix integral
approach, which is powerful~\cite{DFGZJ}, but is not always fully
rigorous.
The Potts model has been addressed via matrix
integrals~\cite{daul,eynard-bonnet-potts,zinn-justin-dilute-potts}. We
refer to~\cite{bernardi-mbm} for a description our current understanding
of this work.

\section{Definitions and notation}
\label{sec:def}

\subsection{Planar maps}
%
A \emph{planar map} is a proper
 embedding of a connected planar graph in the
oriented sphere, considered up to orientation preserving
homeomorphism. Loops and multiple edges are allowed. The \emph{faces}
of a map are the connected components of 
its complement. The numbers of
vertices, edges and faces of a planar map $M$, denoted by $\vv(M)$,
$\ee(M)$ and $\ff(M)$,  are related by Euler's relation
$\vv(M)+\ff(M)=\ee(M)+2$.
 The \emph{degree} of a vertex or face is the number
of edges incident to it, counted with multiplicity. A map is \emm
$m$-valent, if all its vertices have degree $m$. A \emph{corner} is
a sector delimited by two consecutive edges around a vertex;
hence  a vertex or face of degree $k$ defines $k$ corners. The \emph{dual} of a
map $M$, denoted $M^*$, is the map obtained by placing a 
vertex of $M^*$ in each face of $M$ and an edge of $M^*$ across each
edge of $M$; see Figure~\ref{fig:example-map}. 

For counting purposes it is convenient to consider \emm rooted, maps. A map
is {rooted} by orienting an edge, called the \emph{root-edge}.
The origin of this edge is the \emm root-vertex,.
 The  face  that lies to the right of the root-edge
is the \emph{root-face}. In figures, we 
take the root-face
as the infinite face (Figure~\ref{fig:example-map}). 
This explains why we often call the root-face the \emm outer, (or: \emm
infinite,) face, and
its degree the \emm outer degree,. 
The other faces are said to be \emph{finite}. From now on, every {map}
is \emph{planar} and \emph{rooted}. By convention,
we include among rooted planar maps the \emph{atomic\/} map
$m_0$ having one vertex and no edge. The set of rooted planar maps is
denoted $\mM$. 

A map is 
 \emph{separable} if it is atomic
 or can be obtained by gluing two  non-atomic maps  at a vertex. 
Observe that both
 maps with one edge are non-separable.

\begin{figure}[ht!]\begin{center} \input{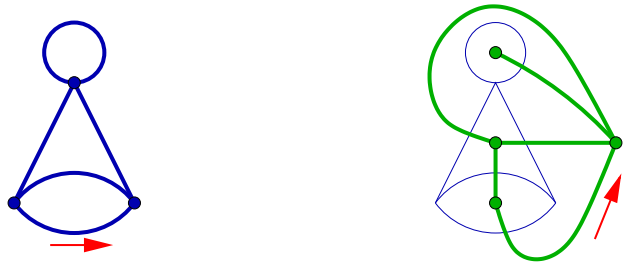}\caption{A rooted planar map and its dual
(rooted at the dual edge).}\label{fig:example-map} \end{center}\end{figure}
%

\subsection{Power series}
Let $A$ be a commutative ring and $x$ an indeterminate. We denote by
$A[x]$ (resp. $A[[x]]$) the ring of polynomials (resp. \fps) in $x$
with coefficients in $A$. If $A$ is a field, then $A(x)$ denotes the field
of rational functions in $x$, and $A((x))$ the field of Laurent series\footnote{A \emm Laurent series, is a series of the form $\sum_{n\ge n_0} a(n)
x^n$, for some $n_0\in \Z$.}
in $x$. These notations are generalised to polynomials, fractions
and series in several indeterminates. We 
denote by bars the reciprocals of variables: that is, $\bx=1/x$, so that $A[x,\bx]$ is the ring of Laurent
polynomials in $x$ with coefficients in $A$.
The coefficient of $x^n$ in a Laurent  series $F(x)$ is denoted
by $[x^n]F(x)$. 
 The \emm valuation, of a Laurent series
$F(x)$ is the smallest $d$ such that $x^d$ occurs in $F(x)$ with a
non-zero coefficient. If $F(x)=0$, then the valuation is $+\infty$.
If $F(x;t)$ is a  power series in $t$ with coefficients in
$A((x))$, that is, a series of the form 
$$
F(x;t)=\sum_{n\ge 0, i\in \Z} f(i;n) x^i t^n,
$$
where for all $n$, almost all coefficients $f(i;n)$ such that $i<0$ are zero, then  the
\emm positive part, of $F(x;t)$ in $x$ is the following series, which
has coefficients in $xA[[x]]$:
$$
[x^>]F(x;t):=\sum_{n\ge 0, i>0} f(i;n) x^i t^n.
$$
We define similarly the 
non-negative part 
of $F(x;t)$ in $x$.

 A power series $F(x_1, \ldots, x_k) \in \GK[[x_1, \ldots, x_k]]$, where $\GK$ is a
field, is \emm algebraic, (over $\GK(x_1, \ldots, x_k)$) if it satisfies a
polynomial equation $P(x_1, \ldots, x_k, F(x_1, \ldots,
x_k))=0$. The series $F(x_1, \ldots, x_k)$ is \emm D-finite, if for all $i\le k$, it satisfies a
(non-trivial) linear differential equation in $x_i$ with 
coefficients in $\GK[x_1, \ldots, x_k]$.
We refer to~\cite{lipshitz-diag,lipshitz-df} for a study of these
series. All algebraic series are D-finite. A series $F(x)$ is \emm
differentially algebraic, if it satisfies a (non-necessarily linear)
differential equation  with coefficients in $\GK[x]$.

\subsection{The Potts model and the Tutte polynomial}
Let $G$ be a graph with vertex set $V(G)$ and edge set $E(G)$.
Let $\nu$ be an indeterminate, and take $q\in \N$. 
A \emm colouring, of the vertices of $G$ in $q$ colours is a map $c : V(G)
\rightarrow \{1, \ldots, q\}$. An edge of $G$ is \emm monochromatic,
if its endpoints share the same colour. Every loop is thus
monochromatic. The number of monochromatic edges is denoted by $m(c)$.
The \emm partition function of the  Potts model,
on $G$ counts colourings by the number of monochromatic edges:
$$
\Ppol_G(q, \nu)= \sum_{c  : V(G)\rightarrow \{1, \ldots, q\}}
\nu^{m(c)}.
$$
The Potts model is a classical magnetism model in statistical physics, which
includes (for $q=2$) the famous Ising model (with no magnetic
field)~\cite{welsh-merino}. Of course, $\Ppol_G(q,0)$ is the chromatic
polynomial of $G$. 

If $G_1$ and $G_2$ are disjoint graphs and $G=G_1 \cup G_2$, then
clearly
\beq\label{Potts-disjoint}
\Ppol_{G}(q,\nu)=\Ppol_{G_1}(q,\nu)\Ppol_{G_2}(q,\nu).
\eeq
If $G$ is obtained by attaching $G_1$ and $G_2$ at one vertex,
then
\beq\label{eq:Potts-1components}
\Ppol_{G}(q,\nu)=\frac 1 q \,
\Ppol_{G_1}(q,\nu)\Ppol_{G_2}(q,\nu).
\eeq

The Potts partition function can be computed by induction on the
number of edges. If $G$ has no edge, then $\Ppol_G(q,\nu)=
q^{|V(G)|}$. Otherwise, let $e$ be an edge of $G$. Denote by $G\backslash
  e$ the graph obtained by deleting $e$, and by ${G\slash e}$ the graph obtained by contracting $e$ (if $e$ is a loop, then it is simply deleted). Then
\beq\label{Potts-induction}
\Ppol_{G}(q,\nu)=  \Ppol_{G\backslash e}(q,\nu)+(\nu-1) \Ppol_{G\slash e}(q,\nu).
\eeq
Indeed, it is not hard to see that $\nu\Ppol_{G\slash e}(q,\nu)$
counts colourings for which $e$ is monochromatic, while $\Ppol_{G\backslash
  e}(q,\nu)- \Ppol_{G\slash e}(q,\nu)$ counts those for which $e$ is
bichromatic.
One important consequence of this induction is that 
  $\Ppol_{G}(q,\nu)$ is
always a \emm polynomial, in $q$ and $\nu$.  We call
it the \emm Potts polynomial, of $G$. Since it is a polynomial, we
will no longer consider $q$
as an integer, but as an indeterminate, and sometimes evaluate
$\Ppol_{G}(q,\nu)$ at  
real values $q$. We also observe that $\Ppol_{G}(q,\nu)$ is a multiple
of $q$: this explains why we will  weight maps by $\Ppol_{G}(q,\nu)/q$.

Up to a change of variables, the Potts polynomial is equivalent to
another, maybe better known, invariant of graphs, namely the \emm Tutte polynomial,
$\Tpol_G(\mu, \nu)$ (see e.g. \cite{Bollobas:Tutte-poly}):
$$
\Tpol_G(\mu,\nu):=\sum_{S\subseteq
  E(G)}(\mu-1)^{\cc(S)-\cc(G)}(\nu-1)^{\ee(S)+\cc(S)-\vv(G)},
$$
where the sum is over all spanning subgraphs of $G$ (equivalently,
over all subsets of edges) and  $\vv(.)$, $\ee(.)$ and $\cc(.)$ denote
respectively the number of vertices, edges and connected
components. For instance, the Tutte polynomial of a graph with no edge
is 1. The equivalence with the Potts polynomial was established  by
Fortuin and Kasteleyn~\cite{Fortuin:Tutte=Potts}: 
\begin{eqnarray}\label{eq:Tutte=Potts}
\Ppol_G(q,\nu)~=~\sum_{S\subseteq E(G)}q^{\cc(S)}(\nu-1)^{\ee(S)}~=~(\mu-1)^{\cc(G)}(\nu-1)^{\vv(G)}\,\Tpol_G(\mu,\nu),
\end{eqnarray}
for $q=(\mu-1)(\nu-1)$. 
In this paper, we work with $\Ppol_G$ rather than $\Tpol_G$ because
we wish to assign real values to $q$ (this is more natural than
assigning real values to $(\mu-1)(\nu-1)$). However,  one
property  looks more natural in terms of $\Tpol_G$: 
if $G$ and $G^*$ are dual 
connected planar graphs  (that is, if $G$ and $G^*$ can be embedded
as dual planar maps) then  
$$
\Tpol_{G^*}(\mu,\nu)=\Tpol_G(\nu,\mu).
$$
Translating this identity in terms of   Potts polynomials thanks
to~\eqref{eq:Tutte=Potts} gives:  
\begin{eqnarray}
\Ppol_{G^*}(q,\nu)&=&q(\nu-1)^{\vv(G^*)-1}\Tpol_{G^*}(\mu,\nu)\nonumber\\
&=&q(\nu-1)^{\vv(G^*)-1}\Tpol_{G}(\nu,\mu)\nonumber\\
&=&\frac{(\nu-1)^{\ee(G)}}{q^{\vv(G)-1}}\Ppol_{G}(q,\mu),
\label{eq:duality-Potts-poly}
\end{eqnarray}
where $\mu=1+q/(\nu-1)$ and the last equality uses Euler's relation:
$\vv(G)+\vv(G^*)-2=\ee(G)$.

\section{Uncoloured planar maps: the recursive approach}
\label{sec:uncoloured-rec}
In this section, we describe the first approach that was used to count
maps: the recursive method. It is based on very simple combinatorial
operations (like the deletion or contraction of an edge), which
translate into non-trivial functional equations defining the
\gfs.  A recent theorem, generalising the so-called \emm quadratic
method,, states that the solutions of all
equations of this type  are algebraic. Since the recursive method
applies to many families of maps, numerous algebraicity results
follow.

\subsection{A functional equation for planar maps}
Consider a rooted planar map, distinct from the atomic map. Delete
the root-edge. If this edge is an isthmus, one obtains two connected components $M_1$ and
$M_2$, and otherwise a single  component $M$,
which we can root in a canonical way (Figure~\ref{fig:map-del}). Conversely, starting from
an ordered pair $(M_1, M_2)$ of maps, there is a unique way to connect
them by a new (root) edge. If one starts instead from a single map
$M$, there are $d+1$ ways to add a root edge, where $d=\df(M)$ is the
degree of the root-face of $M$ (Figure~\ref{fig:map-del2}). 

\begin{figure}[htb]
  \begin{center}
     \scalebox{0.6}{\input{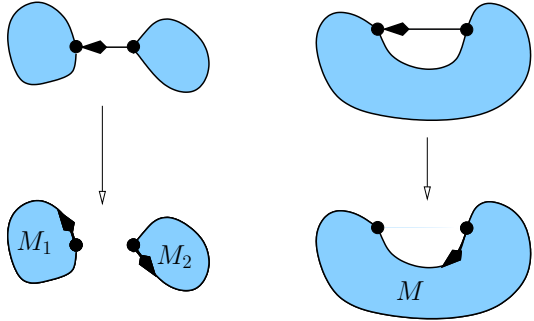}}
  \caption{Deletion of  the root-edge in a planar map.}
\label{fig:map-del}
  \end{center}
\end{figure}

\begin{figure}[htb]
  \begin{center}
  \scalebox{0.6}{\input{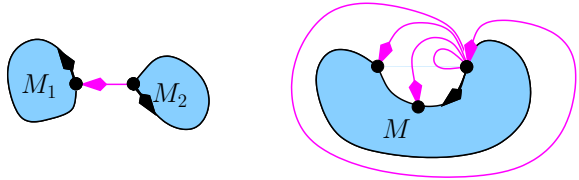}}
  \caption{Reconstruction of a planar map.}
\label{fig:map-del2}
  \end{center}
\end{figure}

Hence, to derive from this recursive description of planar maps a
functional equation for their \gf, we need to take into account the
degree of the root-face, by an additional variable $y$. Hence, let
$$
M(t;y)= \sum_{M\in\mM}t^{\ee(M)}y^{\df(M)}= \sum_{d\ge 0} y^d M_d(t)
$$
be the \gf\ of planar maps, counted by edges and outer-degree.
The series $M_d(t)$ counts  by edges maps with outer degree $d$.
The recursive description of maps  translates as follows:
\begin{eqnarray}
M(t;y)&= &1+ y^2 t  M(t;y)^2 + t\sum_{d\ge 0} M_d(t)(y+y^2+\cdots +
y^{d+1})\nonumber \\
&=& 1+ y^2 t  M(t;y)^2 +  t y \, \frac{yM(t;y)-M(t;1)}{y-1}. \label{1cat-planaires1}
\end{eqnarray}
Indeed, connecting two maps $M_1$ and $M_2$ by an edge produces a map
of outer-degree $\df(M_1)+\df(M_2)+2$, while the $d+1$ ways to add an
edge to a map $M$ such that $\df(M)=d$ produce $d+1$ maps of respective outer degree $1, 2,
\ldots, d+1$, as can be seen on Figure~\ref{fig:map-del2}. The term 1
records the atomic map.

The above equation was first written by Tutte in 1968~\cite{tutte-general}. It is
typical of the type of equation obtained in (recursive) map
enumeration. More examples will be given in Section~\ref{sec:more-func-eq}. One
important feature in this equation is the divided difference
$$
\frac{yM(t;y)-M(t;1)}{y-1},
$$
which prevents us from simply setting $y=1$ to solve for $M(t;1)$
first, and then for $M(t;y)$. The parameter $\df(M)$, and the
corresponding variable $y$, are said to be \emm catalytic, for this
equation --- a terminology borrowed to
Zeilberger~\cite{zeil-umbral}. 

Such equations do not only occur in connection with maps: they also arise
in the enumeration of polyominoes~\cite{bousquet-vcd,feretic-svrtan,temperley}, lattice
walks~\cite{bousquet-petkovsek-1,banderier-flajolet,de-mier,knuth,prodinger},
permutations~\cite{bousquet-stack,bousquet-butler,zeilberger-stack}... The
solution of these equations has naturally attracted some
interest. The  ``guess and check'' approach used in the early 1960s is now replaced by a general method, which we present below in
Section~\ref{sec:quad}.  This method implies in
particular that \emm the solution of  any
(well-founded) polynomial equation with one catalytic variable is
algebraic,. It generalises the \emm
quadratic method, developed by Brown~\cite{brown-square}  for equations of
degree~2 that involve a single additional unknown series (like
$M(t;1)$ in the equation above) and also the \emm kernel method, that applies
to linear equations, and seems to have first appeared in Knuth's {\it  Art of
Computer Programming}~\cite[Section~2.2.1,  Ex.~4]{knuth} (see also~\cite{hexacephale,bousquet-petkovsek-1,prodinger}).

\begin{figure}[b!]
  \begin{center}
       \scalebox{0.6}{\input{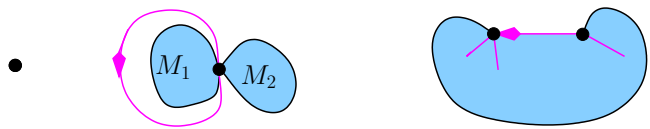}}
    \caption{Contraction of the root-edge in a planar map.}
\label{fig:map-enum-contract}
  \end{center}
\end{figure}

\medskip
\noindent{\bf Contraction vs.\ deletion.}
Before we move to more examples, let us make a simple
observation. Another natural way to decrease the edge number of a map
is to contract the root-edge, rather than  delete it (if this edge is a
loop, one just erases it). When one tries
to use this to count planar maps, one is lead to introduce the degree
of the root-vertex as a catalytic parameter, and a corresponding
variable $x$ in the \gf. 
This yields the same equation as above:
$$
\MM(t;x)= 1+ x^2 t  \MM(t;x)^2 + t\sum_{d\ge 0} \MM_d(t)(x+x^2+\cdots +
x^{d+1}).
$$
As illustrated by Figure~\ref{fig:map-enum-contract}, the term 1 
records the atomic map, the second term corresponds to maps in which the root-edge
is a loop, and the third term to the remaining cases. In particular,
the sum $(x+x^2+\cdots +
x^{d+1})$ now describes how to distribute the adjacent edges when a
new edge is inserted. 
Given that the
contraction operation is the
dual of the deletion operation, it is perfectly natural to obtain the same equation as before. The reason why we mention this alternative
construction is
that, when we establish below a functional equation for maps weighted
by their Potts (or Tutte) polynomial, we will have to use simultaneously these
two operations, as suggested by the recursive description~\eqref{Potts-induction} of
the Potts polynomial. This will naturally result in equations with
\emm two, catalytic
variables $x$ and $y$.

\subsection{More functional equations}
\label{sec:more-func-eq}
The recursive method is extremely robust. We illustrate this  by a few
examples. Two of them --- maps with prescribed face degrees, and
Eulerian maps with prescribed face degrees --- actually cover
infinitely many families of maps.  Some of these examples also have a
colouring flavour. 

\paragraph{Maps with prescribed face degrees.}
Consider for instance the enumeration of \emm triangulations,, that is, maps
in which all faces have degree 3. The recursive deletion of the root-edge gives
 maps in which all finite faces have degree 3, but the outer face may have any degree: these maps are called \emm
near-triangulations,. We denote by $\mNT$ the set of
near-triangulations. The deletion of the root-edge in a near
triangulation gives either two near-triangulations, or a single one,
\emm the outer degree of which is at least two, (Figure~\ref{fig:triang-enum-del}). In both cases, there
is unique way to reconstruct the map we started from.
 Let $\NT(t;y)\equiv \NT(y)$ be the
\gf\ of near-triangulations, counted by edges and by the outer
degree:
$$
\NT(t;y)= \sum_{M\in\mNT}t^{\ee(M)}y^{\df(M)}= \sum_{d\ge 0} y^d \NT_d(t).
$$

\begin{figure}[htb]
  \begin{center}
    \includegraphics[scale=0.6]{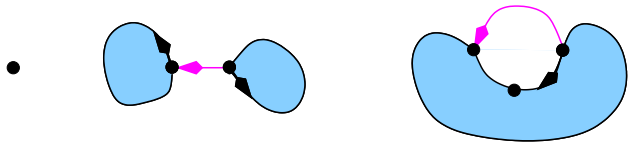}
  \caption{Deletion of the root-edge in a near-triangulation.}
\label{fig:triang-enum-del}
  \end{center}
\end{figure}
The above recursive description translates into
\beq\label{eq:NT}
\NT(y)= 1+ty^2 \NT(y)^2+ t\, \frac{\NT(y)-\NT_0-y\NT_1}{y},
\eeq
where $T_0=1$ counts the atomic map. We have again a divided
difference, this time at $y=0$. Its combinatorial interpretation (``it
is forbidden to add an edge to a map of outer degree 0 or 1'') differs
from the interpretation of the divided difference occurring
in~\eqref{1cat-planaires1}  (``there are multiple ways to add an
edge''). Still, both equations are of the same type and will be solved
by the same method. Note that we have omitted the variable $t$ in the
notation $\NT(y)$, which we will do quite often in this paper, to
avoid heavy notation and enhance the catalytic parameter(s).

Consider now \emm bipartite, planar maps, that is, maps that admit a proper
2-colouring (and then a unique one, if the root-vertex is coloured
white). For planar maps, this is equivalent to saying that all faces have an
even degree. 
Let $\BM(t;y)=\sum_{d\ge 0} B_d(t) y^d$ be the \gf \ of bipartite
maps, counted by edges (variable $t$) and by
half the outer degree (variable $y$). Then the deletion of the root-edge translates
as follows (Figure~\ref{fig:bip}):
\begin{eqnarray}
  \BM(y)&= &1+ ty \BM(y)^2+ t\sum_{d\ge 0} \BM_d (y+y^2+\cdots + y^d)\nonumber
\\
&= &1+ ty \BM(y)^2+ t y \frac{\BM(y)-\BM(1)}{y-1}.\label{eq:bip}
\end{eqnarray}
This is again a quadratic equation with one catalytic
variable, $y$. 

\begin{figure}[htb]
  \begin{center}
     \includegraphics[scale=0.6]{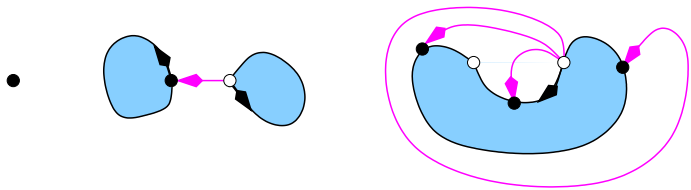}
  \caption{Deletion of the root-edge in a bipartite map.}
\label{fig:bip}
  \end{center}
\end{figure}

\medskip

More generally, it was shown by Bender and
Canfield~\cite{bender-canfield} that the recursive
approach applies to any family of maps for which the
face degrees belong to a given set $D$, provided $D$ differs from a
finite union of arithmetic progressions by a finite set. In all cases,
the equation is quadratic, but may involve more than a single
additional unknown function. For instance, when counting
near-quadrangulations rather than near-triangulations,
Eq.~\eqref{eq:NT} is replaced by
$$
\NQ(y)= 1+ty^2 \NQ(y)^2+ t\, \frac{\NQ(y)-\NQ_0-y\NQ_1-y^2\NQ_2}{y^2},
$$
where $\NQ_i$ counts near-quadrangulations of outer degree $i$.
Bender and Canfield solved
these equations using a theorem of Brown from which the 
quadratic method is derived, proving in particular that the resulting \gf\ is
always algebraic. Their result only involves the edge number, but,
when $D$ is finite, it can be refined by keeping track of the vertex
degree distribution~\cite{mbm-jehanne}.

\paragraph{Eulerian maps with prescribed face degrees.}
A planar map is Eulerian if  all vertices have an
even degree. Equivalently, its faces admit a proper
2-colouring  (and a unique one, if the root-face is coloured
white). Of course, Eulerian maps are the duals of bipartite maps, so
that their \gf\ (by edges, and half-degree of the root-vertex) satisfies~\eqref{eq:bip}. But
 we wish to impose conditions on the face degrees of Eulerian
maps (dually, on the vertex degrees of bipartite maps). 
This includes as a special case the enumeration of (non-necessarily
Eulerian) maps with prescribed face degrees, discussed in the previous
paragraph: indeed, if we require that all
black faces  of
an Eulerian map have degree 2,
each black face can be contracted into a single edge, leaving a standard map
with prescribed (white) face degrees.

Generally speaking, it is difficult to count families of maps with conditions
on the vertex degrees \emm and, on the face degrees (and being
Eulerian is a condition on vertex degrees). 
However,
it was shown in~\cite{mbm-jehanne} that the enumeration of Eulerian
maps such that all black faces have degree in  $D_\bullet$ and all
white  faces have degree in $D_\circ$ can be addressed by the
recursive method when $D_\bullet$ and $D_\circ$ are finite. This is
also true when $D_\bullet=\{m\}$ and $D_\circ=m\N$ (such maps are
called \emm  $m$-constellations,). 


Let us take the example of  Eulerian near-triangulations. All finite
faces have degree 3, while the infinite face, which is white by
convention, has  degree $3d$ for some $d\in \N$. In order to decompose
these maps, we now delete all the edges that bound the black face adjacent
to the root-edge (Figure~\ref{fig:dec-euler-triang}). This leaves 1, 2 or 3 connected components,
which are themselves Eulerian near-triangulations, and which we root
in a canonical way. 
Let $\ET(z;y)\equiv \ET(y)=\sum_{d\ge 0} \ET_d(z) y^d$ be the
\gf\ of Eulerian near-triangulations, counted by black faces
(variable $z$)  and by the outer
degree, divided by 3 (variable $y$). The above decomposition gives:
$$
\ET(y)=1+
 zy\ET(y)^3+ 2z  \ET(y) (\ET(y)-\ET_0) + z(\ET(y)-\ET_0) + z\ \frac{\ET(y)-\ET_0-y\ET_1} y
.
$$
This is a cubic equation with one catalytic variable, which is
routinely solved by the method presented below in Section~\ref{sec:quad}.

\begin{figure}[htb]
  \begin{center}
    \includegraphics[scale=0.6]{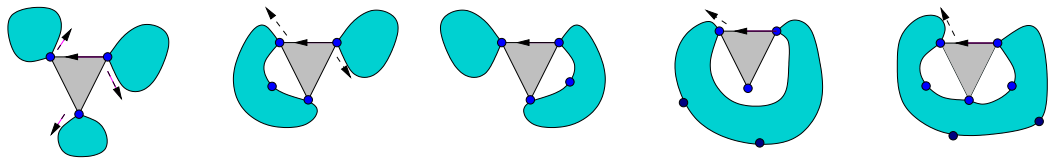}
  \caption{Decomposition of Eulerian near-triangulations.}
\label{fig:dec-euler-triang}
  \end{center}
\end{figure}

The enumeration of Eulerian triangulations is often presented as a
colouring problem~\cite{BDFG02b,DF-Eynard-Guitter}, for the following
reason: a planar triangulation admits a
proper 3-colouring of its vertices if and only if it is (properly) face-bicolourable, that is,
Eulerian\footnote{It is easy to see that the condition is necessary: 
around a face, in clockwise order, one meets either the colours 1, 2,
3 in this order, or 3, 2, 1, and all faces that are adjacent to a
123-face are of the 321-type. The converse is easily seen to hold by
induction on the face number, using
Figure~\ref{fig:dec-euler-triang}.}. Moreover, fixing the colours of the 
endpoints of the root-edge determines completely the colouring. More
generally, let us say that a $q$-colouring is \emm cyclic, if around any
face, one meets either the colours $1,2, \ldots q, 1, 2, \ldots, q$, in
this order, or $q,q-1, \ldots, 1, q, q-1, \ldots ,1 $. Then for $q\ge
3$, a planar
map admits a cyclic $q$-colouring if and only if it is Eulerian and
all its face degrees are multiples of $q$. In this case, it has
exactly $2q$ cyclic colourings. 
The $m$-constellations defined above are of this type (with $m=q$). 

\paragraph{Other families of maps.}
Beyond the two general enumeration problems we have just discussed,
the recursive approach  applies to many other families of planar
maps: loopless
maps~\cite{bender-wormald-loopless,walsh--lehman-III}, maps with
higher
connectivity~\cite{brown-non-separable,brown-tutte-non-separable,gao-5-connected},
dissections of 
a regular
polygon~\cite{brown-triangulations,brown-quadrangulations,tutte-triangulations},
triangulations with large vertex degrees~\cite{bernardi-high-v}, maps
on surfaces of higher
genus~\cite{bender-surface-I,bender-surface-II,gao-surface}...
The resulting equations are often fruitfully combined with \emm composition
equations,  that relate the \gfs\ of two families of maps, for instance
general planar maps and non-separable planar maps
(see. e.g.,~\cite[Eq.~(6.3)]{tutte-census-maps} or~\cite[Eq.~(2.5)]{tutte-triangulations}).

\subsection{Equations with one catalytic variable and algebraicity theorems}
\label{sec:quad}
In this section, we state a general theorem that implies
that the solutions of all  the functional equations we have written so
far are algebraic. We then explain how to solve in practice these equations. The
method extends the \emm quadratic method, that
applies to quadratic equations with a unique additional unknown
series~\cite[Section~2.9]{goulden-jackson}.

Let $\GK$ be a field of characteristic 0, typically $\Q(s_1, \ldots,
s_k)$ for some indeterminates $s_1, \ldots, s_k$.
Let $F(y)\equiv F(t;y)$ be a power series in $\GK(y)[[t]]$, that is, a series in
$t$ with rational coefficients in $y$. Assume that these coefficients 
have no pole at $y=0$. 
The following divided difference (or discrete derivative) is then well-defined:
$$
\Delta F(y) = \frac{F(y)-F(0)}y.
$$
Note that 
$$
\lim _{y\rightarrow 0} \Delta F(y) =F'(0),
$$
where the derivative is taken with respect to $y$. The operator
$\Delta ^{(i)}$ is obtained by applying $i$ times $\Delta$, so that:
$$
\Delta ^{(i)} F(y) = \frac{F(y)-F(0)-yF'(0) -\cdots -
  y^{i-1}/(i-1)!\,F^{(i-1)}(0)}{y^i}. 
$$
Now
$$
\lim _{y\rightarrow 0} \Delta ^{(i)} F(y) =\frac{F^{(i)}(0)}{i!}.
$$
Assume $F(t;y)$ satisfies a functional equation of the form
\beq
\label{main-eq}
F(y)\equiv F(t;y) = F_0(y)+ t\  Q\Big( F(y), \Delta F(y), \Delta
^{(2)}F(y),\ldots ,\Delta ^{(k)}F(y), t; y\Big) ,
\eeq
where $F_0(y)\in \GK(y)$ and $Q(y_0, y_1, \ldots , y_{k}, t; y)$ is  a polynomial in the $k+2$ 
indeterminates $y_0, y_1, \ldots , y_{k}, t$, and a rational function in
the last indeterminate $y$, having  coefficients
 in  $\GK$. 
This equation thus involves, in addition to $F(y)$ itself, $k$
additional unknown series,  namely $F^{(i)}(0)$ for $0 \le i <k$.
%
\begin{theorem} [\cite{mbm-jehanne,bernardi-mbm}]
\label{generic-thm}
Under the above assumptions, the series $F(t;y)$  is algebraic over $\GK(t,y)$.
\end{theorem}

In practice, one proceeds as follows to obtain an algebraic system of
equations defining the $k$ unknown series
$F^{(i)}(0)$. An example will be detailed further down. Write~\eqref{main-eq} in the form
\beq\label{main-eq-pol}
P(F(y), F(0), \ldots, F^{(k-1)}(0), t;y)=0,
\eeq
for some polynomial $P(y_0, y_1, \ldots , y_{k}, t; y)$, and
consider the following equation in $Y$:
\beq\label{eq-dy0}
\frac{\partial P}{\partial y_0} (F(Y), F(0), \ldots, F^{(k-1)}(0),t;Y)=0.
\eeq
On explicit examples, it is usually easy to see that this
equation admits $k$ solutions $Y_0, \ldots, Y_{k-1}$ in the ring of Puiseux series in $t$
with a non-negative valuation (a Puiseux series is a power series in a
fractional power of $t$, for instance a series in $\sqrt t$). By differentiating~\eqref{main-eq-pol}
with respect to $y$, it then follows that
\beq\label{eq-dy}
\frac{\partial P}{\partial y} (F(Y), F(0), \ldots, F^{(k-1)}(0),
t;Y)=0.
\eeq
Hence the following system of $3k$ algebraic equations holds: for
$i=0, \ldots, k-1$,
\begin{eqnarray}
  P(F(Y_i), F(0), \ldots, F^{(k-1)}(0), t;Y_i)&=&0,
\nonumber\\
\frac{\partial P}{\partial y_0} (F(Y_i), F(0), \ldots, F^{(k-1)}(0),
t;Y_i)&=&0,
\label{syst-cat}\\
\frac{\partial P}{\partial y} (F(Y_i), F(0), \ldots, F^{(k-1)}(0),
t;Y_i)&=&0.\nonumber
\end{eqnarray}
This system involves $3k$ unknown series, namely $Y_i$, $F(Y_i)$, and
$F^{(i)}(0)$ for
$0\le i < k$. The fact that the series $F^{(i)}(0)$ are derivatives of
$F$ plays no particular role. Observe that the above system consists of  $k$ times the \emm same, triple of equations, so that
elimination in this system is not obvious~\cite{mbm-jehanne} (and will often end up being
very heavy). When $k=1$,
however, obtaining a solution takes three lines in Maple. 
Consider for instance the equation~\eqref{eq:NT} we have obtained for
near-triangulations. Eq.~\eqref{eq-dy0} reads in this case
$$
Y=t+2tY^3\NT(Y),
$$
and it is clear that it has a unique solution, which is a \fps\ in $t$
with constant term $0$. Indeed, the coefficient of $t^n$ in $Y$ can be
determined inductively in terms of the coefficients of $\NT$.
Then~\eqref{eq-dy} reads 
$$
\NT(Y)=1+3tY^2\NT(Y)^2-t\NT_1.
$$
These two equations, combined with the original equation~\eqref{eq:NT}
taken at $y=Y$, form a system of three polynomial equations involving $Y, \NT(Y)$
and $\NT_1$, from which $Y$ and $\NT(Y)$ are readily eliminated by
taking resultants. This leaves a polynomial equation for the unknown
series $\NT_1$, which counts near-triangulations of outer degree 1:
$$
\NT_1={t}^{2} -27t^5+ 30\,{t}^{3}  \NT_1+t( 1-96\,{t}^{3} ) {\NT_1}^{2}
+ 64\,{t}^{5}{\NT_1}^{3}.
$$
One can actually go further and obtain simple expressions for the
coefficients of $\NT_1$. The above equation admits rational
parametrisations, for instance
$$
t^3= X(1-2X)(1-4X), \quad t\,\NT_1= \frac{X(1-6X)}{1-4X},
$$
and the Lagrange inversion formula yields the number of
near-triangulations of outer degree 1 having $3n+2$ edges (hence $n+2$
vertices) as
$$
\frac{2\cdot 4^n (3n)!!}{n!!(n+2)!},
$$
where $n!!=n(n-2)(n-4)\cdots(n-2\lfloor \frac{n-1}2\rfloor)$.
The existence of such simple formulas will be discussed further, in
connection with bijective approaches (Section~\ref{sec:bij-uncoloured}).


\paragraph{Algebraicity results.} 
The general algebraicity result for solutions of polynomial equations
with one catalytic variable (Theorem~\ref{generic-thm}), combined with the wide applicability of
the recursive method, implies that many families of planar maps have
an algebraic \gf. In the following
theorem, the term \emm\gf,\ refers to the \gf\ by vertices, faces and
edges (of course, one of these statistics is redundant, by
Euler's formula).  
\begin{theorem}[\cite{bender-canfield,mbm-jehanne}]
\label{thm:alg-maps}
  For any set $D\subset \N$ that differs from a finite union of arithmetic progressions by a finite set, the \gf\ of maps such that all faces have their degree in $D$ is algebraic.
If $D$ is finite, this holds has well for the refined \gf\ that keeps track of the number of $i$-valent faces, for all $i\in D$.

For any finite sets $D_\circ$ and $D_\bullet$ in $\N$, the \gf\ of
face-bicoloured maps  such that all white (resp.~black) faces have
their degree in $D_\circ$ (resp.~$D_\bullet$) is algebraic. This
holds as well for the \gf\ that keeps track of the number of
$i$-valent white and black faces, for all $i \in \N$.  

Finally, the \gf\ of face-bicoloured planar maps such that all black faces have degree $m$, and all white faces have their degree in $m\N$, is algebraic.
\end{theorem}

\paragraph{Where is the quadratic method?}
To finish this section, let us briefly sketch why the above procedure
for solving equations with one catalytic variable generalises the
quadratic method.  The first two equations
of~\eqref{syst-cat} show that $y_0=F(Y_i)$ is a double root of $P(y_0, F(0),
\ldots, F^{k-1}(0),t;Y_i)$. Hence $y=Y_i$ cancels the discriminant of $P(y_0, F(0),
\ldots, F^{k-1}(0),t;y)$, taken with respect to $y_0$. When $P$ has
degree 2 in $y_0$, it is easy to see that the third equation of~\eqref{syst-cat} means that $Y_i$ is
actually a double root of the discriminant~\cite[Section~3.2]{mbm-jehanne}: this is the heart of the
quadratic method,   described
in~\cite[Section~2.9]{goulden-jackson}.  That each series $Y_i$ is
a multiple root of the discriminant actually holds for equations of
higher degree,
but this is far from obvious~\cite[Section~6]{mbm-jehanne}.

\section{Uncoloured planar maps: bijections}
\label{sec:bij-uncoloured}
So far, we have emphasised the fact that many families of planar maps
have an algebraic \gf. It turns out that many of them are also
counted by remarkably simple numbers,  which have a strong flavour of
tree enumeration.  Both observations
raise a natural question: is it possible to explain the algebraicity
and/or the numbers  more combinatorially, via bijections
that would relate maps to trees?

We  present in this section two bijections between planar maps
and some families of trees that 
allow one to determine very elegantly the number of planar maps having 
$n$ edges. The first bijection also explains combinatorially why the
associated \gf\ is algebraic. The second one has other virtues, as it
allows to record  the distances of vertices to the
root-vertex. This property has  proved extremely useful is the study of random
maps of large size and their scaling limit~\cite{chassaing-schaeffer,le-gall-maps,le-gall-paulin-bipartite,marckert-mokkadem,marckert-miermont}. Both bijections
could probably qualify as 
\emm Proofs from The Book,~\cite{TheBook}. Both are robust enough to be
generalised to many other families of maps, as discussed in Section~\ref{sec:more-bijections}.

\subsection{Two proofs from The Book?}
\label{sec:thebook}

Both types of bijections involve families of maps with bounded
vertex degrees (or, dually, bounded face degrees). So let us first
recall why planar maps are equivalent to planar 4-valent maps (or
dually, to quadrangulations).

Take a planar map, and create a vertex in the middle of each
edge, called \emm e-vertex, to distinguish it from the vertices of the
original map. Then, turning inside each face, join by an edge each pair of consecutive
e-vertices (Figure~\ref{fig:radial}). The e-vertices, together with these new edges, form a 4-valent
map. Root this map in a
canonical way. This construction is a bijection between rooted planar
maps with $n$ edges and rooted 4-valent maps with $n$ vertices.

\begin{figure}[htb]
  \begin{center}
    \includegraphics[scale=0.6]{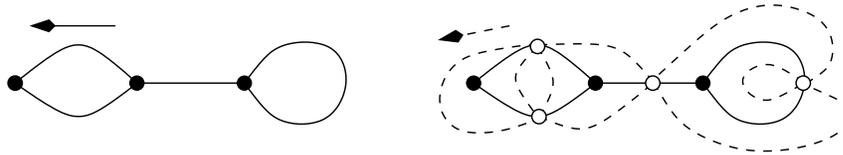}
  \caption{A  planar map with $n$ edges and the corresponding 4-valent
    planar map with $n$ vertices (dashed lines).}
\label{fig:radial}
  \end{center}
\end{figure}

\paragraph{Four-valent maps and blossoming trees.}

The first bijection, due to Schaeffer~\cite{Sch97}, transforms 4-valent maps into \emm blossoming trees,. A
blossoming tree is a (plane) binary tree, rooted at a leaf, such that every
inner node carries, in addition to its two children, a \emm flower,
(Figure~\ref{fig:blossoming}). There are three possible positions for each
flower. If the tree has $n$ inner nodes,
it has  $n$ flowers and $n+2$ leaves. Flowers and leaves are
called \emm half-edges,.

\begin{figure}[htb]
  \begin{center}
    \includegraphics[scale=0.6]{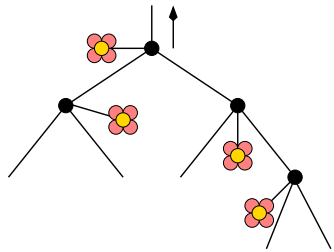}
  \caption{A  blossoming tree with $n=4$ inner nodes.}
\label{fig:blossoming}
  \end{center}
\end{figure}

One obtains  a blossoming  tree by \emm opening, certain edges of a 4-valent
map (Figure~\ref{fig:opening}).
Take a 4-valent map $M$. First, cut the root-edge into two half-edges, that become
leaves: the first of them will be the root of the final tree. Then, start walking  around  the infinite face
in counterclockwise order, beginning with the root edge.   Each time a
non-separating\footnote{An edge is \emph{separating} if its deletion
disconnects the map, \emph{non-separating} otherwise; a map is a tree
if and only if all its edges are separating.}  edge has just been
visited, cut it
into two half-edges: the first one becomes a flower, and the second one,
a leaf. Proceed until all edges are separating edges; this may
require to turn several times around the map.   The final result is
 a blossoming tree, denoted $\Psi(M)$.

\begin{figure}[htb]
  \begin{center}
    \includegraphics[scale=0.5]{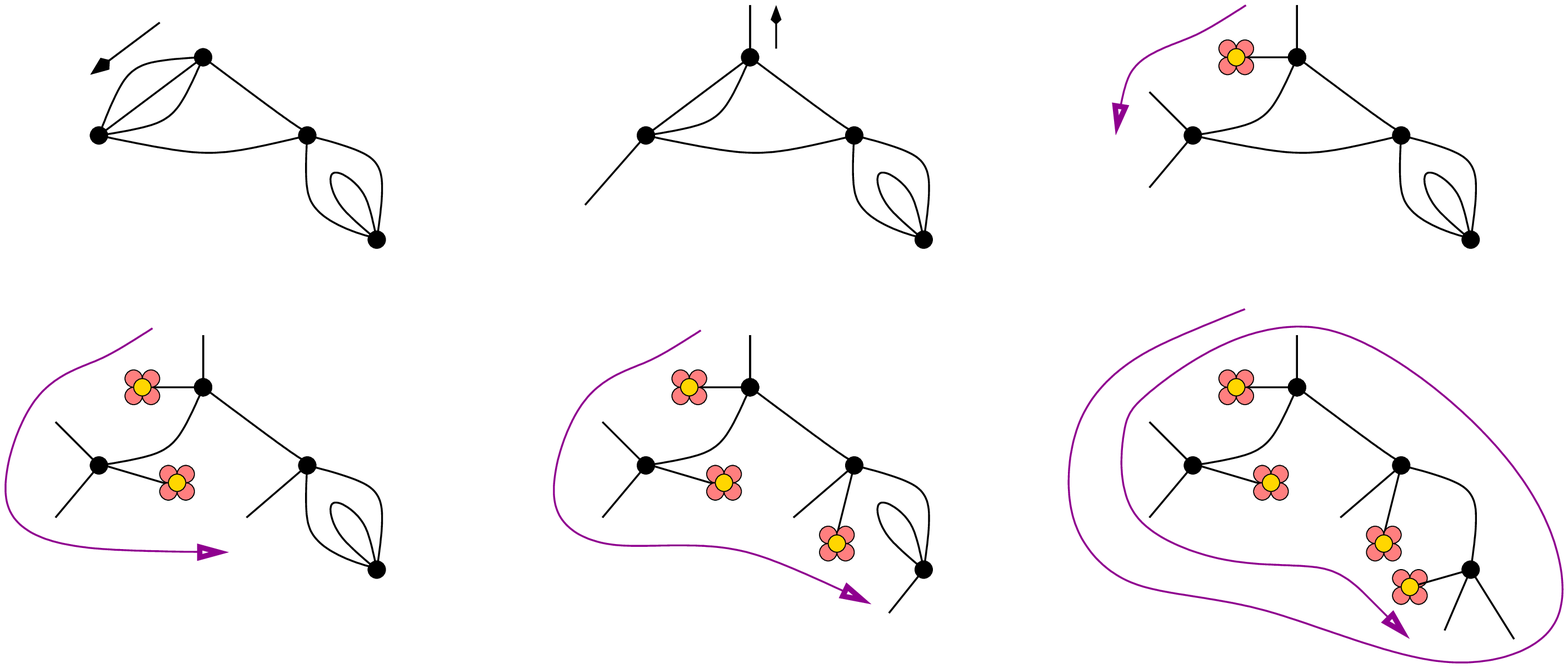}
  \caption{Opening the edges of  a 4-valent
     map gives a blossoming tree.}
\label{fig:opening}
  \end{center}
\end{figure}

Conversely, one can construct a map 
by matching leaves and flowers in
a blossoming tree $T$ as follows (Figure~\ref{fig:unbalanced}, left). Starting from the root, walk around  the infinite
face of $T$ in
counterclockwise order. Each time a flower is immediately followed by a
leaf in the cyclic sequence of half-edges, merge them into an edge in
counterclockwise direction; this creates a new finite face that
encloses no unmatched half-edges. Stop when  all flowers have
been matched. At this point, exactly two leaves remain unmatched
(because there are $n$ flowers and $n+2$ leaves). Observe that the
same two leaves remain unmatched  if one starts walking
around the tree from another position than the root. The
tree $T$ is said to be \emm balanced, if one of the unmatched leaves is the root leaf. 
In this case, 
match it to the other unmatched leaf to form the root-edge of the
map $\Phi(T)$. Again, the complete procedure may require to turn several times around the
tree. We discuss further down what can be done if $T$ is \emm not, balanced.

\begin{prop}[\cite{Sch97}]
  The map $\Psi$ is a bijection between $4$-valent planar maps with
  $n$ vertices and balanced blossoming trees with $n$ inner nodes. Its
  reverse bijection is~$\Phi$.
\end{prop}

With this bijection, it is easy to justify combinatorially the algebraicity of the \gf\ of 4-valent maps.
\begin{cor}\label{coro:4V-alg}
  The \gf\ of $4$-valent planar maps, counted by vertices, is
$$
M(t)=T(t)-tT(t)^3$$
where $T(t)$, the \gf\ of blossoming trees (counted by inner nodes),
satisfies
$$
T(t)=1+3tT(t)^2.
$$
\end{cor}
\begin{proof}
  By decomposing blossoming trees into two subtrees and a flower, it
  should be clear that their \gf\  $T(t)$ satisfies the above equation. Via the bijection $\Psi$, counting maps boils down to counting balanced
  blossoming trees. Their \gf\ is $T(t)-U(t)$, where $U(t)$ counts
  \emm unbalanced, blossoming trees. Consider such a tree, and look at the flower that
  matches the root leaf (Figure~\ref{fig:unbalanced}). This flower is attached to an inner
  node. Delete this node: this leaves, beyond the flower, a 3-tuple of blossoming
  trees. This shows that $U(t)=tT(t)^3$.
\end{proof}

\begin{figure}[htb]
  \begin{center}
    \includegraphics[scale=0.5]{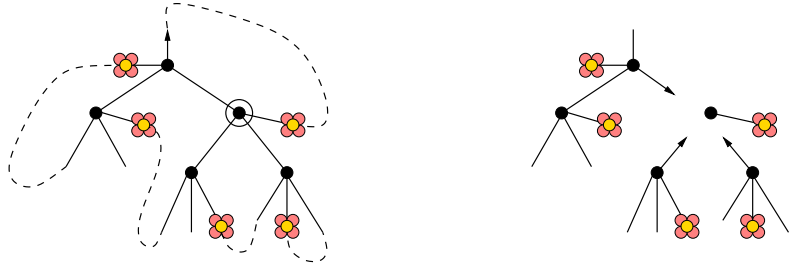}
  \caption{An unbalanced blossoming tree gives rise to three blossoming trees.}
\label{fig:unbalanced}
  \end{center}
\end{figure}

This bijection was originally designed~\cite{Sch97} to explain combinatorially the
simple formulas that occur in the enumeration of maps in which all vertices
have an even degree --- like $4$-valent maps.
\begin{cor}\label{coro:4V-count}
  The number of $4$-valent rooted planar maps with $n$ vertices is
$$
 \frac{2\cdot3^n}{(n+1)(n+2) }{{2n}\choose n}.
$$
\end{cor}
\begin{proof}
  We will prove that the above formula counts balanced blossoming trees of
  size $n$. Clearly, the total number of   blossoming trees of this size is
$$
t_n = \frac{3^n} {n+1 }{{2n}\choose n}
$$
(because binary trees are counted by the Catalan numbers ${{2n}\choose
  n}/(n+1)$). Marking a blossoming tree at one of its two unmatched
leaves is equivalent, up to a re-rooting of the tree, to marking a
balanced blossoming tree at one of its $n+2$ leaves. This shows that
$2t_n=(n+2)b_n$, where $b_n$ counts balanced blossoming trees, and the
result follows. 
\end{proof}

\paragraph{A more general construction.} A variant $\overline \Phi$ of
the above bijection sends pairs 
$(T,\epsilon)$ formed of a
(non-necessarily balanced)
blossoming tree $T$ and of a sign $\epsilon\in\{+,-\}$  onto rooted 4-valent maps with a distinguished
face. This construction works as follows. In the tree $T$, one matches flowers and leaves as
described above. The two unmatched leaves are then used to form the root
edge, the orientation of which is chosen according to the sign~$\epsilon$. This
gives a 4-valent rooted map. One then marks the face of this map located to the
right of the half-edge where $T$ is rooted. For example, the two maps
associated with the (unbalanced) tree of Figure~\ref{fig:unbalanced} are shown in
Figure~\ref{fig:phi-barre}.

\begin{figure}[htb]
  \begin{center}
    \includegraphics[scale=0.5]{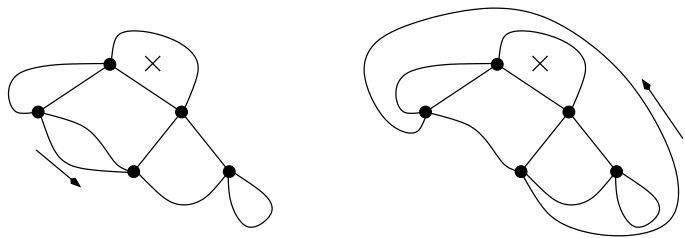}
  \caption{The two maps associated to the tree  of
    Figure~\ref{fig:unbalanced} via the map $\overline \Phi$.}
\label{fig:phi-barre}
  \end{center}
\end{figure}

This construction is
bijective.
Since a 4-valent map with $n$
vertices has $n+2$ faces, it proves that the number
$m_n$ of such maps satisfies $(n+2)m_n= 2\cdot{3^n} {{2n}\choose n}/(n+1)$.

The bijection $\Phi$ described earlier can actually be seen as a specialisation of $\overline\Phi$: If $T$ is  balanced, and one chooses to orient the root
edge of the map in such
a way it starts with the root half-edge of the tree, the map $M$ one obtains
satisfies $\Psi(M)=T$.  The distinguished face is in this case the
root-face, and is thus canonical.

\paragraph{Quadrangulations and labelled trees.}
The second bijection starts from the duals of 4-valent maps, that is,
from quadrangulations. It transforms them into  \emm well labelled
trees,. A \emm labelled tree, is a rooted plane
tree with labelled vertices, such that:

-- the labels belong to $\{1,2, 3, \ldots\}$,

-- the smallest label that occurs is 1, 

-- the labels
of two adjacent vertices differ by $0, \pm 1$.

\noindent
The tree is \emm well labelled, if, in addition, the root vertex has
label 1.

This bijection was first found by Cori \& Vauquelin in 1981~\cite{cori-vauquelin}, but the
simple description we give here was only discovered later by Schaeffer~\cite{chassaing-schaeffer,schaeffer-these}.
As above, there are two versions of this bijection: the most
general one sends 
 rooted quadrangulations with $n$ faces and a distinguished (or:
 \emm pointed,) vertex
 $v_0$ onto
pairs $(T,\epsilon)$ formed of a  labelled tree with $n$ edges $T$ and of
a sign $\epsilon\in\{+,-\}$. Equivalently, it sends  rooted
quadrangulations with a pointed vertex $v_0$ \emm such that the root
edge is oriented away from, $v_0$ (in a sense that will be explained
below) to labelled trees. The
other bijection is a restriction, which sends rooted quadrangulations
(pointed canonically at their root-vertex) onto
well labelled trees. 

So let us describe directly the more general bijection $\overline \Lambda$. Take a
 rooted quadrangulation $Q$ with a pointed vertex $v_0$, such
 that the root-edge is oriented away from $v_0$. By this, we mean that
 the  starting point of the root-edge is closer to $v_0$ than
 the endpoint, in terms of the graph distance\footnote{The fact that
   $Q$ is a quadrangulation, and hence a bipartite map, prevents two
   neighbour vertices to be at the same distance from $v_0$.}. Label all vertices by their
distance to $v_0$. The labels of two neighbours differ by $\pm 1$. If the starting point of the
root-edge has label $\ell$, then the endpoint has label $\ell+1$. The labelling results in two types of
faces: when walking inside a face with the edges on the left, one
sees either a cyclic sequence of labels of the form $\ell, \ell+1,
\ell, \ell+1$, or a sequence of the form $\ell, \ell+1, \ell+2,
\ell+1$.  In the former case, create an edge in the face joining the
two corners labelled $\ell+1$. In the latter one, create an edge from the ``first''
corner labelled $\ell+1$ (in the order described above) to the corner
labelled $\ell+2$. See Figure~\ref{fig:bijection-labelled} for an example. The set of edges created
in this way forms a tree, which we root at the edge created in the
outer face of $Q$, oriented 
away from the endpoint of the root-edge of $Q$ (Figure~\ref{fig:labelled-root}). This tree, $\overline \Lambda(Q)$, contains all
vertices of $Q$, except the marked one.

\begin{figure}[htb]
  \begin{center}
    \includegraphics[scale=0.5]{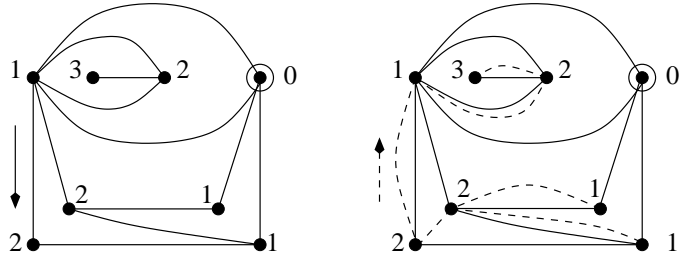}
  \caption{From a rooted quadrangulation with a pointed vertex to a labelled
    tree  (in dashed lines).}
\label{fig:bijection-labelled}
  \end{center}
\end{figure}

\begin{figure}[htb]
  \begin{center}
    \scalebox{0.45}{\input{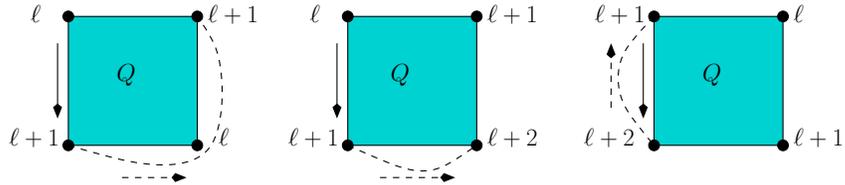}}
   \caption{How to root the tree $\overline \Lambda(Q)$. The only
     vertices that are shown are those of the root-face of $Q$.}
\label{fig:labelled-root}
  \end{center}
\end{figure}

The reverse bijection $\overline V$ works as follows (see Figure~\ref{fig:labelled-inverse} for an example). Start from a labelled tree
$T$.  Create a new vertex $v_0$, away from the tree. Then, visit the  corners of the tree in counterclockwise order. From
each corner labelled $\ell$, send an edge to the next corner labelled
$\ell-1$ (or to $v_0$ if $\ell=1$). This set of edges forms a
quadrangulation $Q$. Each face of $Q$ contains an edge of $T$. Choose
the root-edge of $Q$ in the face containing the root-edge of $T$, according to the
rules of Figure~\ref{fig:labelled-root}.

\begin{figure}[htb]
  \begin{center}
  \scalebox{0.45}{\input{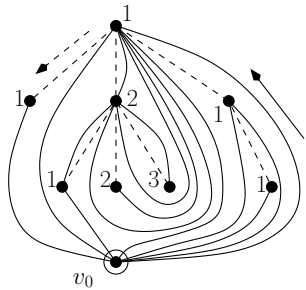}}
  \caption{From   a labelled
    tree  (dashed lines) to a rooted quadrangulation with a pointed vertex.}
\label{fig:labelled-inverse}
  \end{center}
\end{figure}

\begin{prop}
  The map $\overline \Lambda$ sends  bijectively pointed rooted
  quadrangulations with $n$ faces such that the root-edge is oriented away from the
  pointed vertex to labelled trees with $n$ edges. The reverse bijection is  $\overline V$.

When specialised to rooted quadrangulations pointed canonically at their
root-vertex, $\overline \Lambda$ induces a bijection $\Lambda$ between rooted
quadrangulations and well labelled trees.
\end{prop}

Let us now discuss the consequences of this bijection, in terms of
algebraicity and in terms of closed formulas. It is possible to use $\Lambda$ to prove that the \gf\ of rooted quadrangulations is
algebraic~\cite{cori-vauquelin}, and satisfies the system of Corollary~\ref{coro:4V-alg}, 
but this is not as simple as the proof of Corollary~\ref{coro:4V-alg} given above. What \emm
is, simple is to use $\overline \Lambda$ to count quadrangulations,
and hence recover Corollary~\ref{coro:4V-count}. 

This alternative proof works as
follows. First, observe that there are $3^n C_n$ labelled trees with
$n$ edges, where $C_n={{2n}\choose n}/(n+1)$ counts rooted plane trees
with $n$ edges. Indeed, $3^n C_n$ is clearly the number of trees
labelled 0 at the root vertex, such that the labels are in $\Z$ and
differ by at most 1 along edges. If $\ell_0$ denotes the smallest
label of such a tree, adding $1-\ell_0$ to all labels gives a labelled
tree, and this transformation is reversible. Now, the above
proposition implies that  $3^n C_n= (n+2) q_n/2$, where $q_n$ is the
number of quadrangulations with $n$ faces. Indeed, there are $n+2$
ways to point a vertex in such a quadrangulation, and half of these pointings are
such that the root-edge is oriented away from the pointed vertex.

\subsection{More bijections}
\label{sec:more-bijections}
Even though it is difficult to invent bijections, the two bijections
presented above have now been adapted to many other map families, including the two general families described in
Section~\ref{sec:more-func-eq}:   maps with prescribed face
degrees (or, dually, prescribed vertex degrees), and  Eulerian maps
with prescribed face degrees  (dually, bipartite maps with prescribed
vertex degrees). 

On the ``blossoming'' side, Schaeffer's  bijection~\cite{Sch97} was originally
designed, not only for 4-valent maps, but  for maps with
prescribed vertex degrees, \emm provided these degrees are even,. The
case of general degrees was solved (bijectively) a few years later by
a trio of theoretical physicists, Bouttier, Di Francesco and
Guitter~\cite{BDG-planaires}. The  equations they obtain  differ from those
obtained by Bender \& Canfield via the recursive method~\cite{bender-canfield}. See~\cite[Section~10]{mbm-jehanne}
for the correspondence between the two solutions. The case of \emm bipartite, maps with prescribed vertex
degrees was then solved by Schaeffer and the author~\cite{mbm-schaeffer-ising}.
The special case of
 $m$-constellations was solved earlier
 in~\cite{mbm-schaeffer-constellations}.  

On the ``labelled'' side, the extensions of the
Cori-Vauquelin-Schaeffer bijection (which applied to quadrangulations)
to maps with prescribed face degrees, and to \emm Eulerian, maps with
prescribed face degrees, came in a single paper, again due to Bouttier \emm et
al,.~\cite{bouttier-mobiles}.

Other  bijections of the blossoming type exist for certain families of maps
that are constrained, for instance, by higher connectivity conditions, by forbidding
loops, or for dissections of
polygons~\cite{fusy-dissections,PS02,poulalhon-schaeffer}. On the
labelled side, there exist bijections for non-separable
maps~\cite{del-lungo-non-sep,jacquard-schaeffer}, for $d$-angulations
with girth $d$~\cite{bernardi-fusy}, and for maps of
higher genus~\cite{chapuy-marcus-schaeffer}. But the trees are then 
replaced by more complicated objects, namely one-face maps of higher
genus.

\medskip
 All these bijections shed a much better light on
planar maps, by revealing their  hidden tree-like structure. As already
mentioned, they often preserve important statistics, like distances to
the root-vertex. In terms of proving algebraicity results, two
restrictions  should be mentioned:
\begin{itemize}
\item [--] when the degrees are not bounded, it takes a bit of
  algebra to derive, from the system of equations that describes the
  structure of  trees, polynomial equations satisfied by their
  \gfs,
\item [--]  these bijections usually establish the algebraicity of
  the \gf\ of maps that are \emm doubly marked, (like rooted maps with a
  distinguished vertex, or with a distinguished face). The argument
  used to prove Corollary~\ref{coro:4V-alg} has in general no simple
  counterpart. 
\end{itemize}

\section{Coloured planar maps: the recursive approach}
\label{sec:coloured-rec}
We have now reviewed two combinatorial approaches (one recursive, one
bijective) for the enumeration of families of planar uncoloured maps.
We now move to the central topic of this paper, namely the enumeration of
coloured planar maps, and compare both types of problems.

A first simple observation is that algebraicity will no longer be the
rule. Indeed, it has been known for a long time~\cite{mullin-boisees} that the \gf\ of
planar maps, weighted by their number of spanning trees (which is
 the specialisation $\mu=\nu=1$ of the Tutte polynomial~\cite{tutte-dichromate}) is:
$$
\sum_{M\in \mM} t^{\ee(M)} \Tpol_M(1,1)=\sum_{n\ge 0} \frac 1{(n+1)(n+2)} {{2n}\choose
  n}{{2n+2}\choose {n+1}} t^n.
$$
 The asymptotic behaviour of the $n^{\hbox{\small{th}}}$ coefficient, being
$\kappa\, 16^n n^{-3}$, prevents this series from being
algebraic~\cite{flajolet-context-free}. 
The transcendence of this series implies that it cannot be described
by a polynomial equation with one catalytic variable
(Theorem~\ref{generic-thm}). However, it is not difficult to write an equation
with \emm two, catalytic variables for maps weighted by their Tutte
(or Potts) polynomial. This equation is based on the recursive
description~\eqref{Potts-induction}. We present this equation in Section~\ref{sec:eq-func-coloured}, and
another one, for triangulations, in Section~\ref{sec:eq-func-triang}.

The whole point is now to solve equations with two catalytic
variables. Much progress has been made in the past few years on the
\emm linear, case. The equations for coloured maps are not linear, but
they become linear (or quasi-linear, in a sense that will be
explained) for certain special cases, like the enumeration of maps equipped with a spanning
tree or a bipolar orientation. Sections~\ref{sec:bip} and~\ref{sec:trees} are
devoted to these two simpler problems. They show how 
the \emm kernel method,, which was originally designed to  solve
linear equations with one catalytic variable~\cite{hexacephale,bousquet-petkovsek-1,prodinger}, can be extended to
equations with two catalytic variables. Sections~\ref{sec:bip}
and~\ref{sec:trees} actually present two variants of this extension.

We then return to the general case. Following the
complicated approach
used by Tutte to count properly coloured triangulations~\cite{tutte-chromatic-revisited}, we
obtain two kinds of results:
\begin{enumerate}
\item [--] when $q\not = 4$ is of the form $2+2\cos j\pi/m$, for
  integers $j$ and $m$, the \gf\ of $q$-coloured maps satisfies
  also an equation with a \emm single, catalytic variable, and is thus
  algebraic. Explicit results are given for $q=2$ and
  $q=3$;
\item[--] in general, the \gf\ of $q$-coloured maps satisfies a
  non-linear differential equation.
\end{enumerate}
These results, due to Olivier Bernardi and the author~\cite{bernardi-mbm,bernardi-mbm-de}, are presented without proof in  Sections~\ref{sec:beraha}
and~\ref{sec:de}.  We do not make explicit the differential equation
satisfied by the \gf\ of $q$-coloured maps, but give  an (explicit)
\emm system, of differential equations, which we hope to simplify in a
near future.

\subsection{A functional equation for coloured planar maps}
\label{sec:eq-func-coloured}
Let $\mM$ be the set of rooted maps. For $M$ in $\mM$,
recall that  $\dv(M)$ and $\df(M)$ denote respectively
the degrees of the root-vertex and root-face of $M$.
We define the \emm Potts \gf,\ of planar maps by:
\beq\label{potts-planar-def}
\gM(x,y)\equiv \gM(q,\nu,t,w;x,y)
=\frac 1 q \sum_{M\in\mM}t^{\ee(M)}w^{\vv(M)-1}
x^{\dv(M)}y^{\df(M)}
\Ppol_M(q,\nu).
\eeq
Since there is a finite number of maps with a given number of edges, 
and $\Ppol_M(q,\nu)$ is a multiple of $q$,
the generating function $\gM(x,y)$ is a power series in $t$ with
coefficients in $\Q[q,\nu,w,x,y]$. Keeping track of the number of
vertices allows us to go back and forth between the Tutte and Potts
polynomial, thanks to~\eqref{eq:Tutte=Potts}.
 \begin{prop}\label{prop:eq-M}
  The Potts \gf\ of planar maps satisfies:   
\begin{eqnarray}\label{eq:M}
\gM(x,y)&\!\!=\!\!& 1
+xywt\left(qy+(\nu-1)(y-1)\right)\gM(x,y)\gM(1,y)
\nonumber\\
&&+xyt(x\nu-1)\gM(x,y)\gM(x,1)\\
&&
+xywt(\nu-1)\frac{x\gM(x,y)-\gM(1,y)}{x-1}+xyt\frac{y\gM(x,y)-\gM(x,1)}{y-1}.
\nonumber
\end{eqnarray}
\end{prop}
Observe that~\eqref{eq:M} characterises  $\gM(x,y)$ entirely
as a series in $\Q[q,\nu,w,x,y][[t]]$ (think of extracting
recursively the coefficient of $t^n$ in this equation).  
Note also
that when $\nu=1$, then $\Ppol_M(q,\nu) = q^{\vv(M)}$, so that we are
essentially counting planar maps by edges, vertices, and by
the root-degrees $\dv$ and $\df$.  The variable $x$ is no longer
catalytic: it can be set to 1 in the functional equation, which
becomes an equation for $M(1,y)$ with a single catalytic variable~$y$.
\begin{proof}
  This equation is not difficult to establish using the recursive definition
of the Potts polynomial~\eqref{Potts-induction} in terms of deletion
and contraction 
of edges. Of course, one chooses to delete or contract the root-edge
of the map. Let us sketch the proof to see where each term of the
equation comes from.
Equation~\eqref{Potts-induction} gives  
$$
\gM(x,y)=1+\gM_{\backslash }(x,y)+(\nu-1)\gM_{\slash }(x,y),
$$
where the term 1 is the contribution of the atomic map $m_0$, 
$$
\gM_{\backslash }(x,y)=\frac 1 q  \sum_{M\in
  \mM\setminus\{m_0\}}t^{\ee(M)}w^{\vv(M)-1}x^{\dv(M)}y^{\df(M)}
\Ppol_{M\backslash e}(q,\nu),
$$ 
and 
$$
\gM_{\slash }(x,y)=\frac 1 q  \sum_{M\in
  \mM\setminus\{m_0\}}t^{\ee(M)}w^{\vv(M)-1}x^{\dv(M)}y^{\df(M)} 
\Ppol_{M\slash e}(q,\nu),
$$
where $M\backslash e$ and $M\slash e$ denote respectively the maps
obtained from $M$ by deleting and contracting the root-edge $e$.

\noindent {\bf A. The series  $\gM_{\backslash }$.}
We consider the partition
$\mM\setminus\{m_0\}=\mM_1\uplus\mM_2\uplus\mM_3$, where
$\mM_1$  (resp.~$\mM_2$,  $\mM_3$) is the subset of
maps in $\mM\setminus\{m_0\}$ such that the root-edge is 
an isthmus (resp.~a loop, resp.~neither an isthmus nor a loop). 
We denote by  $\gM^{(i)}(x,y)$, for $1\le i
  \le 3$, the contribution of 
$\mM_i$ to the generating function $\gM_{\backslash}(x,y)$, so that
$$
\gM_{\backslash}(x,y)=\gM^{(1)}(x,y)+\gM^{(2)}(x,y)+\gM^{(3)}(x,y).
$$

\noindent {$\bullet$ Contribution of $\mM_1$.}
Deleting the root-edge of a map in $\mM_1$ leaves two maps
$M_1$ and $M_2$, as illustrated 
in Figure~\ref{fig:map-del} (left).  The Potts polynomial of this pair can be
determined using~\eqref{Potts-disjoint}. One thus obtains
$$
\gM^{(1)}(x,y)=qxy^2tw\, \gM(1,y)\gM(x,y),
$$
as the degree of the root-vertex of $M_1$ does not contribute to the degree of
the root-vertex of the final map.

\noindent {$\bullet$ Contribution of $\mM_2$.}
Deleting the root-edge of a map in $\mM_2$ leaves two maps
$M_1$ and $M_2$ attached by their root vertex, as illustrated 
in Figure~\ref{fig:map-enum-contract}.  The Potts polynomial of this pair can be
determined using~\eqref{eq:Potts-1components}. One thus obtains
$$
\gM^{(2)}(x,y)=x^2yt\, \gM(x,1)\gM(x,y),
$$
as the degree of the root-face of $M_1$ does not contribute to the
degree of the root-face of the final map.

\noindent {$\bullet$ Contribution of $\mM_3$.}
Deleting the root-edge of a map in $\mM_3$ leaves a single map
$M$.  If the outer degree of $M$ is $d$, there are $d+1$ ways to add a
new (root-)edge to $M$, as illustrated in Figure~\ref{fig:map-del2} (right). However, a number of
these additions create a loop, and thus their \gf\ must be subtracted. One thus obtains
$$
\gM^{(3)}(x,y)=xt \sum_{d\ge 0}
M_d(q,\nu,t,w;x)(y+y^2+\cdots + y^{d+1} )- xyt\, \gM(x,1)\gM(x,y),
$$
where $M_d(q,\nu,w,t;x)$ is the coefficient of $y^d$ in $M(x,y)$. This
gives
$$
\gM^{(3)}(x,y)=xyt \frac{yM(x,y)-M(x,1)}{y-1}
- xyt\, \gM(x,1)\gM(x,y),
$$
and finally
\begin{multline*}
  \gM_{\backslash}(x,y)=qxy^2tw\, \gM(1,y)\gM(x,y)
+x(x-1)yt\, \gM(x,1)\gM(x,y)
\\+xyt\, \frac{yM(x,y)-M(x,1)}{y-1}.
\end{multline*}

\noindent 
{\bf B. The series  $\gM_{\slash }$.} The study of this series is of
course very similar to the previous one, by duality. One finds:
\begin{multline*}
\gM_{\slash}(x,y)=x^2yt\, \gM(x,1)\gM(x,y)
+xy(y-1)tw\, \gM(1,y)\gM(x,y)\\
+xytw \,\frac{xM(x,y)-M(1,y)}{x-1}.
\end{multline*}

Adding the series $1$,  $\gM_{\backslash }$ and  $(\nu-1)\gM_{\slash }$
gives the functional equation.
\end{proof}

\noindent{\bf Remark.}
 Equation~\eqref{eq:M} is equivalent to an
equation written by Tutte in 1971:
\begin{multline}\label{eq:tM}
  \gtM(x,y)= 1+xyw(y\mu-1)\gtM(x,y)\gtM(1,y)~+~xyz(x\nu-1)\gtM(x,y)\gtM(x,1) \\
+xyw\parfrac{x\gtM(x,y)-\gtM(1,y)}{x-1}+xyz\parfrac{y\gtM(x,y)-\gtM(x,1)}{y-1},
\end{multline}
where $\gtM(x,y)$ counts  maps weighted by their Tutte
polynomial~\cite{Tutte:dichromatic-sums}:
$$
\gtM(x,y)\equiv\gtM(\mu,\nu,w,z; x,y)=
\sum_{M\in\mM}w^{\vv(M)-1}z^{\ff(M)-1}x^{\dv(M)}y^{\df(M)} \Tpol_M(\mu,\nu).
$$
We call the above series \emm the Tutte \gf\ of planar maps.,
 The relation~\eqref{eq:Tutte=Potts} between the Tutte and Potts
polynomials and Euler's relation ($\vv(M)+\ff(M)-2=\ee(M)$) give
$$
\gM(q,\nu,t,w;x,y)=\gtM\left(1+\frac{q}{\nu-1},\nu,(\nu-1)tw,t;x,y\right),
$$
from which~\eqref{eq:M} easily follows.

\subsection{More functional equations}
\label{sec:eq-func-triang}

In a similar fashion, one can write a functional equation for coloured
non-separable planar maps~\cite{liu-non-sep}. It is equivalent
to~\eqref{eq:M} via a simple composition argument~\cite[Section~14]{bernardi-mbm}. Writing equations for coloured maps
with prescribed face degrees is  harder, as the
contraction of the root-edge changes the degree of the finite face
located to the left of the root-edge. This is however not a serious problem if one
counts \emm proper, colourings of triangulations (the faces of degree
2 that occur can be ``smashed'' into a single edge), and in 1971, Tutte
came up with the following equation, the solution of which  kept him busy during
the following decade:
\begin{multline}\label{eq-Tutte}
\gT(x,y)=xy^2q(q-1)+\frac{xz}{yq}\gT(1,y)\gT(x,y)\\
+xz\frac{\gT(x,y)-y^2\gT_2(x)}{y}-x^2yz\frac{\gT(x,y)-\gT(1,y)}{x-1}
\end{multline}
 where $T_2(x)=[y^2]T(x,y)$.
The series $T(x,y)$ defined by this equation  is
\beq\label{T-def}
T(x,y)= \sum_{T} z^{\ff(T)-1}x^{\dv(T)} y^{\df(T)} \Ppol_T(q,0),
\eeq
where the sum runs over all non-separable near-triangulations (maps
in which all finite faces have degree 3). Note that the number of edges and the number of vertices of $T$ can be obtained from $\ff(T)$ and $\df(T)$, using
\beq\label{autres-param}
\vv(T)+\ff(T)=2+\ee(T) \quad \hbox{and} \quad 2\ee(T)= 3(\ff(T)-1)+\df(T).
\eeq

There seems to be no straightforward extension to the Potts
\gf\footnote{Although another type of functional equation is given
  in~\cite[Sec.~3]{eynard-bonnet-potts} for the Potts \gf  \ of cubic
  maps, using matrix integrals.}, and it is not
until  recently that an equation was obtained for the Potts
\gf\  $\gQ(x,y)$ of \emm quasi-triangulations,~\cite{bernardi-mbm}. We refer to
that paper for the precise definition of this class of maps, which is
not so important here. What \emm is, important is that the series $\gQ(0,y)$ is the Potts \gf\ of near-triangulations:
$$
\gQ(0,y)\equiv
\gQ(q,\nu,t,z;0,y)=
\frac 1 q \sum_{T\in\mNT}
t^{\ee(T)}
z^{\ff(T)-1}y^{\df(T)}
\Ppol_T(q,\nu).
$$

\begin{prop}\label{prop:eq-Q}
The Potts \gf\ of quasi-triangulations satisfies 
\begin{multline}
 \label{eq:Q}
Q(x,y) = 1 
+ zt\, \frac{Q(x,y)-1-yQ_1(x)}{y}+  xzt (Q(x,y)-1) + xyzt Q_1(x) Q(x,y)
\\
+yzt(\nu-1)\gQ(x,y)(2x\gQ_1(x)+\gQ_2(x))
+y^2t\left(q+ \frac{\nu-1}{1-xzt\nu}\right) \gQ(0,y)\gQ(x,y)
\\+\frac{yt(\nu-1)}{1-xzt\nu} \frac{\gQ(x,y)-\gQ(0,y)}{x}
  \end{multline}
where $\gQ_1(x)=[y]\gQ(x,y)$ and $\displaystyle
\gQ_2(x)=[y^2]\gQ(x,y)=\frac{(1-2xzt\nu)}{zt\nu}\gQ_1(x)$.
\end{prop}
Tutte's equation~\eqref{eq-Tutte} for non-separable, properly coloured
near-triangulations can be recovered from this
proposition~\cite[Section~14]{bernardi-mbm}. In Section~\ref{sec:trees}, we
will use the following (equivalent) equation for the \gf\ of
quasi-triangulations, weighted by their Tutte polynomial:
\begin{multline}\label{eq:Tutte-triang}
  \gtQ ( x,y ) =
1+tz \, \frac {
 \gtQ ( x,y ) -1-y\gtQ_1 ( x ) 
}{y}+xtz \left( \gtQ ( x,y ) -1\right)
 +xytz  \gtQ_1 ( x )   \gtQ ( x,y )\\
+tz{y}(\nu-1)\gtQ ( x,y ) \left( 2x\gtQ_1 ( x ) +\gtQ_2 ( x )\right)
\\
+{y}^{2}t\left(\mu\, +{\frac {{t}x\nu\,z }{1-x
\nu\,tz}}\right)\gtQ ( 0,y )\gtQ ( x,y ) 
+\frac{yt}{1-x\nu\,tz } {\frac {\gtQ ( x,y ) -\gtQ ( 0,y )   }{ x}},
\end{multline}
where $\gtQ_1(x)=[y]\gtQ(x,y)$ and $\displaystyle
\gtQ_2(x)=[y^2]\gtQ(x,y)=\frac{(1-2xzt\nu)}{zt\nu}\gtQ_1(x)$.
We call the specialisation $\gtQ(0,y)$  \emm the Tutte \gf\ of
near-triangulations:,
$$
\gtQ(0,y)\equiv
\gtQ(q,\nu,t,z;0,y)=
\sum_{T\in\mNT}
t^{\ee(T)}
z^{\ff(T)-1}y^{\df(T)}
\Tpol_T(\mu,\nu).
$$

\subsection{A linear case: bipolar orientations of maps}
\label{sec:bip}
Let $G$ be a connected graph with a root-edge $(s,t)$. A \emm bipolar
orientation, of $G$ is an acyclic orientation of the edges of $G$ such that
$s$ is the single source and $t$ the single sink.  Such orientations
exist if and only if $G$ is non-separable.  
It is known~\cite{greene-zaslavsky,lass-orientations} that the number of  bipolar
orientations of  $G$  is:
$$
(-1)^{\vv(G)} \frac{\partial \Ppol_G}{\partial q} (1,0).
$$
This number is also called the \emm chromatic invariant, of
$G$~\cite[p.~355]{Bollobas:Tutte-poly}.
This expression implies that  the \gf\  of  (non-atomic) planar maps equipped with a bipolar orientation,
counted by  edges ($t$),  non-root vertices ($w$), degree of the root-vertex ($x$)
and of the root-face ($y$) is
$$
B(t,w;x,y)= - \frac{\partial }{\partial q} \left. \big(q\gM(q,0,t,-w;x,y)-q\big)\right|_{q=1}
=- \frac{\partial \gM}{\partial q} (1,0,t,-w;x,y),
$$
where $\gM(x,y)$ is the Potts \gf\ of planar maps, defined
by~\eqref{potts-planar-def} (we have used $\gM(1,0,t,w;x,y)=1$).
By differentiation, it is easy to derive from~\eqref{eq:M} an equation
with two catalytic variables satisfied by $B(x,y)$. Using again
$\gM(1,0,t,w;x,y)=1$, this equation is found to be \emm linear,: 
\beq\label{eq:bip-M}
 \left( 1+
\frac{xy  tw}{1-x}+\frac{xy  t}{1-y}
\right)  \gB  ( x,y )
=
x{y}^{2}wt+{\frac {{x}^{2}ywt }{1-x}}\gB  ( 1,y )
+{\frac {x{y}^{2}t  }{1-y}}\gB  ( x,1 ).
\eeq

Similarly,  using~\eqref{autres-param}, one finds that the \gf\  of
planar near-triangulations equipped with a bipolar orientation, 
counted by  non-root faces  ($z$), degree of the root-vertex ($x$)
and of the root-face ($y$) is
$$
\gBT(z;x,y)= - \frac{\partial \gT}{\partial q} (0,i
z;x,i y),
$$
where $i ^2=-1$ and $\gT(q,z;x,y)$ is Tutte's \gf\ for coloured non-separable
near-triangulations, defined by~\eqref{T-def}.
Given that $\gT(1,z;x,y)=0$, the equation satisfied by $\gBT$ is again
linear:
\beq\label{eq:bip-T}
\left(1 -{\frac {xz}{y}}-{\frac {z{x}^{2}y}{x-1}} \right) \gBT ( x,y
)
=
 x{y}^{2}-xzy\gBT_2 ( x ) -{\frac {z{x}^{2}y \gBT ( 1,y ) }{x-1}}
\eeq
with $\gBT_2 ( x )= [y^2] \gBT(x,y)$.

In the past few years, much progress has been made in the solution
of linear equations with two catalytic variables~\cite{mbm-kreweras,mbm-motifs,mbm-petkovsek2,mbm-mishna,Mishna-jcta,Mishna-Rechni}. It is now understood
that a certain group of rational transformations, which leaves
invariant the \emm kernel, of the equation (the
coefficient of $\gB(x,y)$) plays an important role. In particular,
when this group is finite, the equation can often be solved in an
elementary way, using what is sometimes called the \emm algebraic
version, of the kernel method~\cite{mbm-kreweras,mbm-mishna}. This is
the case  for  
both~\eqref{eq:bip-M} and~\eqref{eq:bip-T}. We detail the solution of~\eqref{eq:bip-T}, and explain how to adapt it to solve~\eqref{eq:bip-M}.
 \begin{prop}\label{prop:bip-triang}
 The number of bipolar orientations of near-triangulations   having $m+1$
vertices is
$$
\frac{(3m)!}{(4m^2-1)m!^2(m+1)!}.
$$
 The number of bipolar orientations of near-triangulations   having $m+1$
vertices and  a root-face of degree $j$  is
\beq\label{bip-sol-T}
\frac{j(j-1)(3m-j-1)!}{m!(m+1)!(m-j+1)!}.
\eeq
 For $m\ge 2$, the number of bipolar orientations of
 near-triangulations   having $m+1$ 
vertices,  a root-vertex of degree $i$ and  a root-face of degree $j$  is
 $$
\frac{(i-1)(j-1)(2m-j-2)!(3m-i-j-1)!}
{(m-1)!m! (m-j+1)!(2m-i-j+1)!} \left( (2j+i-6)m+i+3j-j^2-ij\right).
$$
The corresponding \gfs\ in $1$, $2$ and $3$ variables are D-finite.
\end{prop}
The last two formulas  are due to Tutte~\cite[Eqs.~(32)
  and~(34)]{lambda12}.  He also derived them from the functional
equation~\eqref{eq:bip-T}, but his proof involved a lot of guessing, while
ours is constructive. The first formula in
Proposition~\ref{prop:bip-triang} seems to be new.
\begin{proof}
 It will prove convenient
to set $x=1/(1-u)$. After multiplying~\eqref{eq:bip-T} by $(x-1)/x^2/y$,
the equation we want to solve reads:
\beq\label{eq-BY-uy}
  u \by 
\left(1-u- z ( y\bu +\by ) \right)
\gBT \left( \frac 1{1-u},y \right)  
=uy-R(u)-S(y),
\eeq
with $\bu=1/u$, $\by=1/y$, $R(u)=zu\gBT_2 \left( \frac 1{1-u} \right)$ and $S(y)=z \gBT ( 1,y )$.
 Let $K(u,y)=1-u- z ( y\bu +\by ) $ be the \emm kernel , of
this equation.  This kernel is invariant by the
transformations: 
$$
\Phi: (u,y) \mapsto (yz\bu,y) \quad \hbox{and} \quad
\Psi: (u,y) \mapsto (u,u\by).
$$
Both transformations are involutions, and, by applying them
iteratively to $(u,y)$, one obtains 6
pairs $(u',y')$ on which $K(\cdot, \cdot)$ 
takes the same value:
\beq\label{6id}
(u,y) {\overset{\Phi}{\longrightarrow}}  (yz\bu, y) {\overset{\Psi}{\longrightarrow}}  (yz\bu, z\bu)  {\overset{\Phi}{\longrightarrow}} (z\by, z\bu) {\overset{\Psi}{\longrightarrow}}  (z\by, u\by) {\overset{\Phi}{\longrightarrow}}  (u, u\by) {\overset{\Psi}{\longrightarrow}} (u,y) .
\eeq
For each such pair $(u',y')$, the corresponding specialisation
of~\eqref{eq-BY-uy} reads
$$
 ( {u'}/ {y'})\, K(u,y)\gBT(1/(1-u'),y')=u'y'- R(u')-S(y').
$$
 We form the alternating sum of the 6 equations of this form obtained from the
 pairs~\eqref{6id}. The series $R(\cdot)$ and  $S(\cdot)$ cancel out,
 and, after dividing by $K(u,y)$,  we obtain:
\begin{multline*}
u\by \gBT\left  ( \frac 1{1-u},y \right)
-z\bu \gBT\left ( \frac 1{1-yz\bu}, y\right)
+y\gBT\left ( \frac 1{1-yz\bu}, z\bu\right)
\\-u\by\gBT\left ( \frac 1{1-z\by}, z\bu\right)
+z\bu\gBT\left ( \frac 1{1-z\by}, u\by\right)
-y\gBT\left ( \frac 1{1-u}, u\by\right)
 \\
=
\frac{uy-y^2z\bu+ yz^2\bu^2- z^2\by\bu+zu\by^2-u^2\by}{1-u- z (y\bu +\by   )}
.
\end{multline*}
The above identity holds in the ring of \fps\ in $z$ with
coefficients in $\Q(u,y)$, which we consider as a sub-ring of
Laurent series in $u$ and $y$.
Recall that  $\gBT(x,y)$ has coefficients in $xy^2\Q[x,y]$. Hence, in
the left-hand side of this identity, the terms with positive exponents
in $u$ and $y$ are exactly those of $u\by \gBT\left  ( \frac 1{1-u},y
\right)$.
It follows that  the latter series is the positive part (in $u$
and $y$) of the rational function
$$
R(z;u,y):=\frac{uy-y^2z\bu+ yz^2\bu^2- z^2\by\bu+zu\by^2-u^2\by}{1-u-
  z (y\bu +\by   )}.
$$
It remains to perform a coefficient extraction. One first finds:
$$
\frac{1}{1-u-z(\by+y\bu)}= \sum_{n\ge 0} \sum_{a\ge -n}\sum_{b=-n}^n 
z^n u^a y^b 
{n\choose {\frac{b+n}2}} {{\frac{b+n}2+n+a}\choose n}
$$
where the sum is restricted to triples $(n,a,b)$ such that $n+b$ is even. 
An elementary calculation then yields the expansion of $R(z;u,y)$, and finally
\beq\label{gBT-sol}
\gBT\left(z;\frac 1{1-u},y\right)=   
\sum_{n\ge 0} \sum_{i\ge 0}\sum_{j=2}^{n+2} 
z^n u^i y^j \frac{(i+1)(j-1)(i+j) \left( \frac{3n+j}2 +i-1\right)!}
{\left(\frac{n-j}2+1\right)! \left( \frac{n+j}2 +i+1\right)! 
\left( \frac{n+j}2\right)!},
\eeq
where the sum is restricted to triples $(n,i,j)$ such that $n+j$ is
even.
In particular, the case $u=0$ shows that the number of bipolar
orientations of near-triangulations having $n$ finite faces and a
root-face of degree $j$ is 
$$
 \frac{(j-1)j \left( \frac{3n+j}2 -1\right)!}
{\left(\frac{n-j}2+1\right)! \left( \frac{n+j}2 +1\right)! 
\left( \frac{n+j}2\right)!},
$$
which coincides with~\eqref{bip-sol-T}, given that the number of vertices of
such maps is $1+(n+j)/2$.

The first formula of the proposition is then obtained by summing over
$j$. The third one is easily verified using~\eqref{gBT-sol}. However,
it can also be \emm derived, from~\eqref{gBT-sol} if one prefers a 
constructive proof. One proceeds as follows. First, observe that if a
rational function $R(u)$ is of the form $P(1/(1-u))$, for
some Laurent polynomial $P$, then $P(x)$ coincides with
the  expansion of 
$R(1-\bx)$ as a Laurent series in $\bx$. In particular, if $P(x)\in
x\Q[x]$, then  $P(x)$ is the
positive part in $x$ of the expansion of $R(1-\bx)$ in $\bx$.  The coefficient of $z^n
y^j$ in $\gBT(z;x,y)$ is precisely in $x\Q[x]$, so that we can apply
this extraction procedure to the right-hand side of~\eqref{gBT-sol}. 
We first express the  coefficient of $z^n y^j$ as a rational function
of $u$, using
$$
\sum_{i\ge 0} u^i {{a+b+i}\choose a}= 
\frac 1{u^b (1-u)^{a+1}} -\sum_{j=0}^{b-1} \frac 1 {u^{b-j}}{{a+j}\choose a}.
$$
Then, we set $u=1-\bx$, expand  in $\bx$ this rational function, and
extract the positive part in $x$. For the above series, this gives:
\begin{eqnarray*}
  [x^>]\sum_{i\ge 0} u^i {{a+b+i}\choose a}
&=&
[x^>]\frac{x^{a+1}}{(1-\bx )^b} -\sum_{j=0}^{b-1}
    \frac 1 {(1-\bx)^{b-j}}{{a+j}\choose a}
\\
&=&
[x^>]\frac{x^{a+1}}{(1-\bx )^b}\\
&=&\sum_{k= 0}^a x^{a+1-k} {{k+b-1}\choose k}.
\end{eqnarray*}
Combining these two ingredients yields the third formula of the
proposition.

Finally, the form
of  these three formulas, together with the closure properties of
D-finite series~\cite{lipshitz-diag,lipshitz-df}, imply that the associated \gfs\ are D-finite.
\end{proof}

\medskip
The same method allows us to solve the linear equation~\eqref{eq:bip-M}  obtained
for bipolar orientations of general maps. One sets $x=1+u$ and
$y=1+v$. It is also convenient to write
$$
\gB(x,y)= xy^2tw+ x^2y^2t^2w\, G(x,y).
$$
The equation satisfied by $G$ reads
\begin{multline*}
  uv \left(1-t (1+\bu)(1+\bv)(u+vw)\right) G(1+u,1+v)=\\
uv-tu(1+u)G(1+u,1)-twv(1+v)G(1,1+v).
\end{multline*}
The relevant transformations $\Phi$ and $\Psi$ are now
$$
\Phi: (u,v) \mapsto (\bu w v,v) \quad \hbox{and} \quad
\Psi: (u,v) \mapsto (u,u \bv \bw).
$$
Again, they generate a group of order 6. 
 One finally obtains that $G(1+u,1+v)$ is the non-negative part
(in $u$ and $v$) of the following rational function:
$$
\frac{(1-\bu\bv)(u\bv-w\bu)(\bu v-\bv\bw)}
{1-t (1+\bu)(1+\bv)(u+vw)}.
$$
A coefficient extraction, combined with Lemma~6 of~\cite{mbm-motifs},
yields the following results.
\begin{prop}\label{prop:bip}
For $1\le m <n$,  the number of bipolar orientations of  planar maps   having $n$
 edges and $m+1$ vertices
  is
$$
\frac 2{(n-1)n^2}   {{n} \choose {m-1}}{{n} \choose {m}}
{{n} \choose {m+1}}.
$$
For  $1\le m <n$ and $2\le j\le m+1$,  the number of bipolar orientations of  planar maps   having $n$ edges, $m+1$ vertices
and  a root-face of degree $j$  is
$$
 \frac{ j(j-1)}{(n-1)n^2}  {{n}\choose {m}}{{n}\choose {m+1}}{{n-j-1}\choose {m-j+1}}.
$$
For $n\ge 3$, $1\le m <n$, $2\le i \le n-m+1$ and $2\le j\le m+1$,
the number of bipolar orientations of  planar maps   having $n$
edges, $m+1$ vertices, 
a root-vertex of degree $i$ and a root-face of degree $j$  is
$$
  \frac{(i-1)(j-1)}{(n-1)n} {n \choose m}  
\left[  {n-j-1 \choose n-m-2}{ n-i-1 \choose m-2}-
{n-j-1 \choose n-m-1}{ n-i-1 \choose {m-1}}\right].
$$
The associated \gfs\ are D-finite.
\end{prop}
The solution we have sketched is very close
to~\cite[Section~2]{mbm-motifs}. Eq.~\eqref{eq:bip-M} was also solved
independently by Baxter~\cite{baxter-dichromatic}, but his solution
involved some guessing, while the one we have presented here is constructive.

\subsection{A quasi-linear case: spanning trees}
\label{sec:trees}
When $\mu=\nu=1$, the Tutte polynomial $\Tpol_G(\mu,\nu)$ gives
 the number of spanning trees of $G$. The equations~\eqref{eq:tM}
and~\eqref{eq:Tutte-triang} that define the
Tutte \gfs\ of our main two families of planar maps turn out to
be much easier to solve in this case.

Consider first general planar maps, and Tutte's
equation~\eqref{eq:tM}. We replace $w$ by $wt$ and $z$ by $zt$ so that
$t$ keeps track of the edge number. When $\mu=\nu=1$, the equation reads:
\begin{multline}\label{eq:tM-11}
 \left(1-\frac{x^2ywt}{x-1}- \frac{xy^2zt}{y-1}-xyzt(x-1)\gtM(x,1)-xywt(y-1)\gtM(1,y) \right)  \gtM(x,y)= \\1 
-\frac{xyzt}{y-1}\gtM(x,1)-\frac{xywt}{x-1}\gtM(1,y).
\end{multline}
Observe that, up to a factor $(x-1)(y-1)$, the \emm same, linear combination of $\gtM(x,1)$ and
$\gtM(1,y)$ appears on the right- and left-hand sides. This
property was observed, but not fully exploited,  by
Tutte~\cite{tutte-dichromatic-revisited}. Bernardi~\cite{bernardi-pc} showed that it allows us to solve~\eqref{eq:tM-11} using the
standard kernel method usually applied to \emm linear, equations with two
catalytic variables~\cite{mbm-motifs}. We 
thus obtain a new proof of the following result, due to Mullin~\cite{mullin-boisees}.
 Using Mullin's
 terminology, we say that a map equipped with a
 distinguished spanning tree is \emm tree-rooted,.

\begin{prop}\label{prop:tree-rooted}
 The number of tree-rooted planar maps with $n$ edges is
$$
\frac{(2n)!(2n+2)!}{n!(n+1)!^2(n+2)!}.
$$  
The number of  tree-rooted planar maps with $i+1$ vertices and $j+1$ faces
  is
$$
\frac{(2i+2j)!}{i!(i+1)!j!(j+1)!}.
$$
The associated \gfs\ are D-finite.
\end{prop}
\begin{proof}
  Set
$$
S(u,v)\equiv S(w,z,t;u,v)= \frac 1{(1-ut)(1-vt)}\gtM\left(
w,z,t^2;\frac 1 {1-ut}, \frac 1{1-vt}\right).
$$
Eq.~\eqref{eq:tM-11} can be rewritten as
\beq\label{eq:S-planar}
\left(1-t(u+v+w\bu+z\bv)\right) S(u,v)= \left(1-uvt^2S(u,v)\right)
\left(1-tz\bv S(u,0)-tw\bu S(0,v)\right).
\eeq
Observe that the (Laurent) polynomial
$\left(1-t(u+v+w\bu+z\bv)\right)$ is invariant by the transformation
$u\mapsto w\bu$. Seen as a  polynomial in $v$, it has two
roots. Exactly one of them, denoted 
$V\equiv V(w,z,t;u)$ is a \fps\ in $t$ with coefficients in
$\Q[w,z,u,\bu]$, satisfying 
\beq\label{V-eq}
V= t\left(z+(u+w\bu) V+V^2\right).
\eeq
In~\eqref{eq:S-planar}, specialise $v$ to $V$. The left-hand side
vanishes, and hence the right-hand side vanishes as well. Since its
first factor is not zero, there holds
$$
tzu S(u,0)+twV S(0,V)=uV.
$$
Now replace $u$ by $w\bu$ in~\eqref{eq:S-planar}, and specialise again
$v$ to $V$. This gives
$$
tzw\bu S(w\bu,0)+twV S(0,V)=w\bu V.
$$
By taking the difference of the last two equations, we obtain
\beq\label{trees-diff}
tzu S(u,0)-tzw\bu S(w\bu,0)= (u-w\bu) V.
\eeq
Since $S(u,0)$ is a series in $t$ with coefficients in $\Q[w,z,u]$,
this equation implies that $tzuS(u,0)$ is the positive part in $u$ of
$(u-w\bu) V$. The number $\TR(i,j)$ of tree-rooted planar maps having $i+1$
vertices and $j+1$ faces is the coefficient of $w^iz^j t^{i+j}$ in
$\gtM(w,z,t;1,1)$, that is, the coefficient of $w^iz^{j}
t^{2i+2j}$ in $S(u,0)$, or the coefficient of $w^iz^{j+1} t^{2i+2j+1}u$
in $tzuS(u,0)$. The Lagrange inversion formula, applied
to~\eqref{V-eq}, yields
\beq\label{lagrange-trees}
[w^iz^jt^nu^{n+1-2i-2j}]V= \frac{(n-1)!}{i!(j-1)!j!(n+1-i-2j)!}.
\eeq
Hence
\begin{eqnarray*}
  \TR(i,j)&=& [w^iz^{j+1} t^{2i+2j+1}u] \left(tzuS(u,0)\right)
\\
&=& [w^i z^{j+1} t^{2i+2j+1} u^0] V - [w^{i-1} z^{j+1} t^{2i+2j+1}
  u^2]V
\hskip 10mm \hbox{(by~\eqref{trees-diff})}
,
\end{eqnarray*}
which, thanks to~\eqref{lagrange-trees}, gives the second result of the proposition. The first one
follows by summing over all pairs $(i,j)$ such that
$i+j=n$. Alternatively, one can apply the Lagrange inversion formula
to the equation satisfied by $V$ when $w=z=1$, which is
$V=t(1+uV)(1+\bu V)$. 
\end{proof}

Similarly, we can derive from the functional equation~\eqref{eq:Tutte-triang} defining the
Tutte-\gf\ of quasi-triangulations the number of tree-rooted
near-triangulations having a fixed outer degree and number of
vertices. This result is also due to Mullin~\cite{mullin-boisees}, but
the proof is new.

\begin{prop}\label{prop:tree-rooted-triang}
  The number of tree-rooted near-triangulations having $i+1$ vertices and a
  root-face of degree $d$ is
$$
\frac d{(i+1)(4i-d)}{{3i-d}\choose i} {{4i-d} \choose i} .
$$
The associated \gf\ is D-finite.
\end{prop}
\begin{proof}
  We specialise to $\mu=\nu=z=1$ the equation~\eqref{eq:Tutte-triang} that defines the
  Tutte-\gf\ of quasi-triangulations. We then replace $\gtQ_2(x)$ by its expression in terms of $\gtQ_1$.
Again, the same linear combination of $\gtQ_1(x)$ and $\gtQ(0,y)$ occurs  in
the right- and left-hand sides, and the equation can be rewritten as
\begin{multline}\label{Q-S}
  \left( 1-t\by -xy(1-tx)- \frac{ty}{x(1-tx)}\right)\gtQ(x,y)
=\\
\left( 1-t\by -tR_1(x)- \frac{ty}{x(1-tx)} \gtQ(0,y)\right)
\left(1-xy\gtQ(x,y)\right)
\end{multline}
where $R_1(x)=x+\gtQ_1(x)$. Let us denote $u:=x(1-xt)$. Equivalently, we introduce a new indeterminate $u$
and set
$$
x= X(u):=\frac{1-\sqrt{1-4ut}}{2t}.
$$
The (Laurent) polynomial $\left(1-t\by -uy- {t\bu y}\right)$ occurring in the
left-hand side of~\eqref{Q-S} is invariant by the
transformation $y \mapsto tu\by /(t+u^2)$. As a polynomial in $u$, it has two roots. One
of them is a power series in $t$ with constant term $0$, satisfying
\beq\label{U-eq}
U= t\,\frac{y+U\by}{1-Uy}.
\eeq
In~\eqref{Q-S}, specialise $x$ to $X(U)$. The left-hand side vanishes,
leaving
$$
tU R_1(X(U))+ {ty} \gtQ(0,y)=U(1-t\by).
$$
If we first replace $y$ by $t\by U/(t+U^2)=t/(1-t\by)$ in~\eqref{Q-S} before
specialising $x$ to $X(U)$, we obtain instead
$$
  tU R_1(X(U))+ \frac{t^2}{1-t \by}\
  \gtQ\left(0,\frac{t}{1-t\by}\right)=t\by U.
$$
By taking the difference of the last two equations, one finds:
\begin{eqnarray*}
{ty} \gtQ(0,y)- \frac{t^2}{1-t\by}\
\gtQ\left(0,\frac{t}{1-t\by}\right)
&=& U(1-2t\by). 
\end{eqnarray*}
Since $\gtQ(0,y)$ is a series in $t$ with coefficients in $\Q[y]$,
this equation implies that ${ty} \gtQ(0,y)$ is the positive part in $y$ of
$U(1-2t\by)$.  The Lagrange inversion formula, applied
to~\eqref{U-eq}, gives:
$$
[t^ny^{3i-n+2}] U=\frac 1 n {n\choose {i+1}}{{n+i-1}\choose i}.
$$
This yields
\begin{eqnarray*}
[t^n y^{3i-n}] \gtQ(0,y)&=&\frac 1 {n+1} {{n+1}\choose {i+1}}{{n+i}\choose i}
-\frac 2 n {n\choose {i+1}}{{n+i-1}\choose i}\\
&=&
\frac{3i-n}{(i+1)(n+i)}{n\choose i}{n+i\choose i},
\end{eqnarray*}
which is equivalent to the proposition, as a near-triangulation with
$n$ edges and outer degree $3i-n$ has $i+1$ vertices.
 \end{proof}

\subsection{When $q$ is a Beraha number: Algebraicity}
\label{sec:beraha}
We now report on more difficult results obtained recently by Olivier
Bernardi and the author~\cite{bernardi-mbm} by following and adapting
Tutte's enumeration 
of properly $q$-coloured triangulations~\cite{tutte-chromatic-revisited}.

For certain values of $q$, it is possible to derive from the equation
with two catalytic variables defining $M(x,y)$ an equation with a
single catalytic variable (namely, $y$) satisfied by $M(1,y)$. For instance, one can derive from
the case $q=1$ of~\eqref{eq:M} that $M(y)\equiv M(1,y)$ satisfies
$$
M(y)= 1+ y^2 t\nu w  M(y)^2 + \nu t y \, \frac{yM(y)-M(1)}{y-1}.
$$
This is only a moderately exciting result, as the latter equation is
just the standard functional equation~\eqref{1cat-planaires1} obtained by
deleting recursively  the root-edge in planar
maps.  

But let us be persistent.  When $q=2$, one can derive
from~\eqref{eq:M} that $M(y)\equiv M(1,y)$ satisfies a 
polynomial equation with one catalytic variable, involving two
additional unknown series, namely
$M(1)$ and $M'(1)$. This equation  is rather big
(see~\cite[Section~12]{bernardi-mbm}), and we do not write it here. No
combinatorial way to derive it is known at the moment. 
When $\nu=0$, the series $M'(1)$ disappears, and one recovers the
standard equation~\eqref{eq:bip} obtained by
deleting recursively  the root-edge in  bipartite planar maps.

This construction works as soon as $q\not=0,4$ is of the form $2+2\cos
(j\pi/m)$, for integers $j$ and $m$.  These numbers
generalise  \emm Beraha's numbers, (obtained for $j=2$), 
which occur frequently in connection with  chromatic properties
 of planar
 graphs~\cite{beraha-kahane-weiss,fendley-chromatic,jacobsen-richard-salas,jacobsen-salas,martin,saleur}. They
 include the three integer values $q=1, 2, 3$. Given
 that the solutions of polynomial equations with one catalytic
 variable are always algebraic (Theorem~\ref{generic-thm}), the following algebraicity
 result holds~\cite{bernardi-mbm}.

\begin{theorem}\label{thm:alg-planaires}
 Let $q\not= 0,4$ be of the form
$2 +2 \cos j \pi/m$ for two integers $j$ and $m$.
Then the series $M( q, \nu, t, w; x,y)$, defined by~\eqref{eq:M}, 
is algebraic over $\Q(q, \nu,t,w,x,y)$. 
\end{theorem}

A similar method works for quasi-triangulations.
\begin{theorem}\label{thm:alg-triang}
 Let $q\not = 0,4$ be of the form $2+2\cos j\pi/m$ for two integers
 $j$ and $m$.
Then the series $\gQ(q, \nu, t ,z ; x,y)$, defined by~\eqref{eq:Q}, 
is algebraic over $\Q(q, \nu, t, z,x,y)$. 
\end{theorem}

For the two integer values  $q=2$ (the Ising model) and
$q=3$ (the 3-state Potts model), we have applied the procedure
described in Section~\ref{sec:quad} to obtain explicit
algebraic equations satisfied by $\gM(q, \nu,t,1;1,1)$ and $\gQ(q,
\nu,t,1;1,1)$. However, when $q=3$, we could only solve the case
$\nu=0$ (corresponding to proper colourings). The final equations are remarkably
simple. We give them here for  general planar maps. For
triangulations, these equations are not new: the Ising model on
triangulations was already solved  by several other methods (including
bijective ones, see Section~\ref{sec:bij-ising} for details), and properly
3-coloured triangulations are just Eulerian triangulations, as
discussed in Section~\ref{sec:more-func-eq}. With the help of Bruno
Salvy, we have also  conjectured an algebraic equation of degree
11 for the \gf\ of properly 3-coloured cubic maps (maps in which all vertices
have degree 3). By the duality relation~\eqref{eq:duality-Potts-poly},
this corresponds to the
series $\gQ(q,\nu,t,1;1,1)$ taken at $q=3$, $\nu=-2$. 

\begin{theorem}
\label{thm:2}
The Potts \gf\ of planar maps $M(2,\nu,t,w;x,y)$, defined
by~\eqref{potts-planar-def} and taken at $q=2$, is algebraic.
 The specialisation $M(2,\nu,t,w;1,1)$ 
has degree  $8$ over  $\Q(\nu, t,w)$. 

When $w=1$, the degree decreases to
 $6$, and the equation admits a rational parametrisation.
Let $S\equiv S(t)$ be the unique power series in $t$ with constant
term $0$ satisfying
$$
S=t\; \frac{\left( 1+3\,\nu\,S-3\,\nu\,{S}^{2}-{\nu}^{2}{S}^{3} \right)
  ^{2}}
{
  1-2\,S+2\,{\nu}^{2}{S}^{3}-{\nu}^{2}{S}^{4}  }.
$$
Then
\begin{multline*}
  M(2, \nu, t,1;1,1)= \frac{
  1+3\,\nu\,S-3\,\nu\,{S}^{2}-{\nu}^{2}{S}^{3}}
{\left(1-2\,S+2\,{\nu}^{2}{S}^{3}-{\nu}^{2}{S}^{4}\right)^2}\times
\\
\left(
{\nu}^{3}{S}^{6}
+2\, {\nu}^{2} (1- \nu ){S}^{5}
+\nu\, ( 1-6\,\nu ) {S}^{4}
-\nu\, ( 1-5\,\nu ) {S}^{3}
+ (1+ 2\,\nu ) {S}^{2}
-(3+ \nu ) S
+1
\right)
.
\end{multline*}
\end{theorem}

\begin{theorem}\label{thm:3}
The Potts \gf\ of planar maps $M(3,\nu,t,w;x,y)$, defined
by~\eqref{potts-planar-def} and taken at $q=3$, is algebraic.

The specialisation  $M(3,0,t,1;1,1)$ 
that counts \emm properly, three-coloured planar maps by edges,
has degree $4$ over  $\Q(t)$, and  admits a rational parametrisation.
Let $S\equiv S(t)$ be the unique power series in $t$ with constant
term $0$ satisfying
$$
t= \frac{ S(1-2\,S^3)  }{\left( 1+2S \right)  ^{3}}.
$$
Then
$$
  M(3,0, t,1;1,1)= 
{\frac { \left(1+ 2\,S \right)  
\left(1 -2\,{S}^{2}-4\,{S}^{3}-4\,{S}^{4} \right) }
{ \left(1- 2\,{S}^{3} \right) ^{2}}}.
$$
\end{theorem}

\subsection{The general case: differential equations}
\label{sec:de}
The culminating, and final point in Tutte's study of properly
coloured triangulations was a non-linear differential equation
satisfied by their \gf. For the more complicated problem of
counting maps weighted by their Potts polynomial, we have come with a
\emm system, of differential equations that defines the corresponding
\gf~\cite{bernardi-mbm-de}. 
One compact way to write this system is as follows.
\begin{theorem}\label{thm:de-maps}
Let $\beta=\nu-1$ and
$$
\Delta(t,v)=(q\nu+\be^2)-q (\nu+1 ) v+ (\be t ( q-4 )  ( wq+\be ) +q)v^2.
$$  
There exists a unique triple $(A(t,v), B(t,v),C(t,v))$ of
  polynomials in $v$ with coefficients in $\Q[q,\nu,w][[t]]$, having
  degree $4, 2$ and $2$ respectively in $v$, such that
$$
  \begin{array}{ll}
    A(0,v)=(1-v)^2,& \quad A(t,0)=1,\\
B(0,v)= 1-v, &\quad C(t,0)= w(q+2\beta)-1-\nu,
 \end{array}$$
and
\beq\label{de}
\frac 1{C(t,v)}\frac{\partial }{\partial v} \left( \frac{v^4
  C(t,v)^2}{A(t,v) \Delta(t,v)^2}\right)
= \frac {v^2}{B(t,v)}\frac{\partial }{\partial t} \left( \frac{
  B(t,v)^2}{A(t,v) \Delta(t,v)^2}\right).
\eeq
Let $A_i(t)$ (resp.~$B_i(t)$) denote the coefficient of $v^i$ in
$A(t,v)$ (resp.~$B(t,v)$). Then the Potts \gf\ of planar maps,
$M(1,1)\equiv M(q,\nu, t, w;1,1)$, defined
by~\eqref{potts-planar-def},  is related to $A$
and $B$ by
\begin{multline*}
  12\,{t}^{2}w \left( q\nu+{\beta}^{2} \right) M ( 1,1 ) 
-A_2 ( t ) +2\,B_2 ( t ) 
-8\,t \left( w(q+2\beta)-\nu-1 \right) B_1( t )+  B_1 ( t  ) ^{2} 
\\= 4\,t \left( 1-3\,
 ( \beta+2 ) ^{2}t
+ \left( 6\, ( \beta+2)  ( q+2\,\beta ) t+q+3\,\beta \right) w
-3\,t ( q+2\,\beta ) ^{2}{w}^{2}\right).
\end{multline*}
\end{theorem}
\noindent{\bf Comments}\\
1.  Let us write 
$$
A(t,v)=\sum_{j=0}^4 A_j(t) v^i, \quad B(t,v)=\sum_{j=0}^2 B_j(t) v^i,
\quad C(t,v)=\sum_{j=0}^2 C_j(t) v^i.
$$
The differential equation~\eqref{de} then translates into a system of
9 differential equations (with respect to $t$) relating the 11 series
$A_j, B_j, C_j$. However, $A_0(t)=A(t,0)$ and $C_0(t)=C(t,0)$ are given
explicitly as initial conditions, so that there are really as many unknown
series in $t$ as differential equations. Observe moreover that  no derivative of the
series $C_j(t)$ arise in the system. This is why we only need initial conditions for
the series $A_j(t)$ and $B_j(t)$. They are prescribed by the
values of $A(0,v)$ and $B(0,v)$.

\noindent
2. The form of the above result is very close to Tutte's solution of
properly coloured planar triangulations, which can be stated as in
Theorem~\ref{thm:Tutte-triang} below. However, Tutte's case is
simpler, as it boils 
down to only 4 differential equations. This explains why Tutte could
derive from his system a \emm single, differential equation for
the \gf\ of properly  coloured triangulations. More
precisely, it follows from the theorem below
 that, if  $t=z^2$ and $H\equiv H(t)= t^2T_2(1)$,
\beq\label{Tutte-ED}
2q^2(1-q)t +(qt+10H-6tH')H''+q(4-q)(20H-18tH'+9t^2H'')=0.
\eeq
So far, we have  not been able to
derive from Theorem~\ref{thm:de-maps} a single differential equation for coloured
planar maps. 
\begin{theorem}\label{thm:Tutte-triang}
Let
$$
\Delta(v)=v+4-q.
$$  
There exists a unique pair $(A(z,v), B(z,v))$ of
  polynomials in $v$ with coefficients in $\Q[q][[z]]$, having
  degree $3$ and $1$ respectively in $v$, such that
$$
  \begin{array}{ll}
    A(0,v)=1+v/4,& \quad A(z,0)=1,\\
 B(0,v)= 1,&
 \end{array}$$
and
$$
-\frac {4z}{v}\frac{\partial }{\partial v} \left( \frac{v^3
  }{A(z,v) }\right)
= \frac {1}{B(z,v)\Delta(v)}\frac{\partial }{\partial z} \left( \frac{
  B(z,v)^2}{A(z,v) }\right).
$$
Let $A_i(z)$ (resp.~$B_i(z)$) denote the coefficient of $v^i$ in
$A(z,v)$ (resp.~$B(z,v)$). Then the  face \gf\ of properly $q$-coloured planar
near-triangulations having outer-degree $2$,
denoted $T_2(q,z;1)$ and defined by~\eqref{eq-Tutte}, is related to $A$ and $B$ by
\begin{multline*}
20\,{z}^{4} (q -4 ) T_2 (q, z;1 )/q -2\,  B_1 ( z )   ^{2}
- \left( 96\,{z}^{2}-24\,{z}^{2}q+1 \right) B_1 ( z ) 
+2\,A_2 ( z ) 
\\
-2\,{z}^{2} \left(10-q+ 432\,{z}^{2}-216\,{z}^{2}q+27\,{z}^{2}{q}^{2}
\right) 
=0.
\end{multline*}
\end{theorem}

\section{Some bijections for coloured planar maps}
\label{sec:bij-coloured}

Certain specialisations of the Potts \gf\ of planar maps can be
determined using a purely bijective approach. 
\subsection{Bipolar orientations of maps} 
The numbers that arise in the enumeration of bipolar orientations of
planar maps (Proposition~\ref{prop:bip}) are known to count other families of objects: \emm Baxter
permutations, (not the same Baxter as in~\cite{baxter-dichromatic}!),
pairs of \emm twin trees,, and
certain configurations of \emm non-intersecting lattice paths.,
Several bijections have been established recently between
bipolar orientations and these
families~\cite{bonichon-mbm-fusy,felsner-bipolar,fusy-bipolar}. Let us
mention, however, that the only
family that is simple to enumerate is that of  non-intersecting
lattice paths (via the Lindstr\"om-Gessel-Viennot theorem). Hence
bijections with this family are the only ones that really provide a self-contained
proof of Proposition~\ref{prop:bip}.

Regarding bipolar orientations of triangulations
(Proposition~\ref{prop:bip-triang}), we are currently working on certain bijections  with \emm Young
tableaux, of height at most 3, in collaboration with Nicolas Bonichon
and \'Eric Fusy.
\subsection{Spanning trees}
It is not hard to count in a  bijective manner tree-rooted maps with
$i+1$ vertices and $j+1$ faces (Proposition~\ref{prop:tree-rooted}).
The construction below, which is usually attributed to Lehman and
Walsh~\cite{walsh--lehman-II},  is actually not  far 
from Mullin's original proof~\cite{mullin-boisees}. Starting from a
tree-rooted map $(M,T)$, one walks around the tree $T$ in counterclockwise order,
starting from the root-edge of $M$ (Figure~\ref{fig:tree-rooted}, left), and:
\begin{itemize}
\item [--] when an edge $e$ of $T$ is met, one  walks along this edge, and
writes $a$ when $e$ is met for the first time, $\bar a$ otherwise;
\item [--] when an edge $e$ not in $T$  is met, one crosses the edge, and writes 
$b$ when $e$ is met for the first time, $\bar b$ otherwise.
\end{itemize}
This gives a shuffle of two Dyck words\footnote{A Dyck word on the
  alphabet $\{a,\bar a\}$  is a word that contains as many occurrences
  of $a$ and $\bar a$, and such that every prefix contains at least as
  many $a$'s as $\bar a$'s. A shuffle of two Dyck words can be seen as
a walk in the first quadrant of $\Z^2$, starting and ending at the origin.} $u$ and $v$, one of length $2i$ on the
alphabet $\{a,\bar a\}$ (since there are $i$ edges in $T$), one
of length $2j$  on the alphabet $\{b,\bar b\}$  (since there are $j$
edges not in the tree). The number of such shuffles is
$$
{{2i+2j}\choose {2i}} C_i C_j,
$$
where $C_i={{2i}\choose i}/(i+1)$ is the number of Dyck words of
length $2i$. The construction is easily seen to be bijective, and this
gives the second result of Proposition~\ref{prop:tree-rooted}.

As already explained, the first result of
Proposition~\ref{prop:tree-rooted} follows by summing over all 
$i,j$ such that $i+j=n$. A \emm direct, bijective proof 
was only obtained in 2007 by Bernardi~\cite{bernardi-boisees}. It
transforms  a tree-rooted 
map into a pair formed of a plane tree and a non-crossing
partition. See~\cite{bernardi-chapuy-boisees} for a recent extension
to maps of higher genus.

\medskip
Mullin's original  construction~\cite{mullin-boisees} decouples the tree-rooted map $(M,T)$
into two objects (Figure~\ref{fig:tree-rooted}, right):
\begin{itemize}
\item [--] a plane tree with $j$ edges, which is the dual of $T$ and
corresponds to the Dyck word $v$  on $\{b,\bar b\}$ described above,
\item [--]  a plane tree $T'$, which consists of $T$ and of  $2j$
  half-edges; this tree can be seen as the Dyck word $u$ shuffled with the word $c^{2j}$. 
\end{itemize}
The vertex degree
distribution of  $M$ coincides with the degree
distribution of $T'$, and Mullin used this property to count tree-rooted maps
with prescribed vertex degrees (or dually, with prescribed face
degrees, since a map and its dual have the same number of spanning
trees). Indeed,  it is easy to count trees with
a prescribed degree distribution~\cite[Thm.~5.3.10]{stanley-vol-2}. In particular, the number of plane trees with
root-degree $d$, such that $n_k$ non-root vertices have degree $k$,
for $k\ge 1$,  
and carrying in addition $2j$ half-edges,  is
$$
T'(d,j,n_1, n_2, \ldots):=\frac {d (2j-1+ \sum _k n_k)!}{(2j)! \prod_k n_k !},
$$
so that the number of tree-rooted
maps $(M,T)$ in which the root-vertex has degree $d$ and  $n_k$
non-root vertices have degree $k$, for $k\ge 1$,  is 
$$
\frac 1 {j+1}{{2j}\choose j} T'(d,j,n_1, n_2, \ldots)=\frac {d (2j-1+
  \sum _k n_k )!}{j!(j+1)! \prod_k n_k !},
$$
where $j= \ee(M)-\vv(M)+1= (d+\sum_k (k-2)n_k)/2$ is the \emm excess,
of $M$ (and also the number of faces, minus 1).
In particular,  the number of tree-rooted
maps $(M,T)$ having a root-vertex of degree $d$  and $2i-d$ non-root vertices of degree 3 is
$$
\frac{d(4i-d-1)!}{i!(i+1)! (2i-d)!},
$$
since such maps have excess $i$. This is the dual statement of
Proposition~\ref{prop:tree-rooted-triang}.

\begin{figure}[htb]
  \begin{center}
    \includegraphics[scale=1.2]{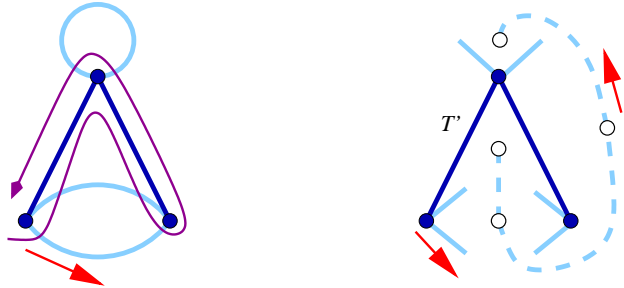}
  \caption{Left: The tour of a tree-rooted map gives an encoding by a
    shuffle of Dyck words, here $bbaa\bar b\bar b\bar a b\bar b \bar
    a$.
Right: Alternatively, one can decouple a tree-rooted map into the dual
  plane tree (dashed lines) and a plane tree $T'$ carrying half-edges.}
\label{fig:tree-rooted}
  \end{center}
\end{figure}

\subsection{The Ising model ($q=2$)}
\label{sec:bij-ising}
As observed in~\cite{mbm-schaeffer-ising}, a simple transformation relates the Potts
\gf\ of maps at $q=2$ to the enumeration of bipartite maps by vertex
degrees.
\begin{prop}
  Let $B(t,v,w;x)$ be the \gf\ of planar bipartite maps, counted by
   edges ($t$), non-root vertices of degree $2$ (variable $v$), non-root vertices of
  degree $\not = 2$ (variable $w$), and degree of the
  root-vertex ($x$). Let $\gM(q,\nu,t, w;x,y)$ be the Potts \gf\ of
  planar maps, defined by~\eqref{potts-planar-def}. Then
$$
\gM\left(2, tv, \frac t{1-t^2v^2},w;x,1\right) = B(t,v+w,w;x).
$$
This identity can be refined by keeping track of the number of non-root
vertices of each degree and colour. Let
$$
\overline M(\nu, t,x_1, x_2, \ldots, y_1, y_2, \ldots; x)
= \sum_{M}\nu^{m(M)}t^{\ee(M)}x^{\dv(M)}\prod_{i\ge
  1}x_i^{\vv_i^\circ(M)}y_i^{\vv_i^\bullet(M)}
,
$$
where  the sum runs over all $2$-coloured maps $M$ rooted at a black
vertex, $m(M)$ is the number of monochromatic edges in $M$, and
$\vv_i^\circ(M)$ (resp.~$\vv_i^\bullet(M)$) is the number of
non-root white (resp.~black) vertices of degree $i$. Let
$$
\overline B(t,x_1, x_2, \ldots, y_1, y_2, \ldots; x)
= \sum_{M}t^{\ee(M)}x^{\dv(M)}\prod_{i\ge
  1}x_i^{\vv_i^\circ(M)}y_i^{\vv_i^\bullet(M)}
,
$$
where the sum runs over all bipartite maps, properly bicoloured in
such a way the root-vertex is black. Then
$$
\overline M\left(tv,\frac t{1-t^2v^2}, x_1, x_2, \ldots, y_1, y_2, \ldots ;
x\right)
=\overline B(t, x_1, v+x_2, x_3,\ldots, y_1, v+y_2, y_3,\ldots; x).
$$
\end{prop}
\begin{proof}
  We establish directly the second identity, which implies the first
  one by specialising each $x_i$ and $y_i$ to $w$.
Take a 2-coloured planar map $M$, rooted at a black vertex.
On each edge, add a (possibly empty) sequence of square
vertices of degree 2, in such a way the resulting map is properly bicoloured. An
example is shown on Figure~\ref{fig:ising-bip}.  Every
monochromatic edge receives an odd number of square vertices, while
every dichromatic edge receives an even number of these vertices. Each
addition of a  vertex of degree 2 also results in the addition of an
edge. Since $M$ can be recovered from the bipartite map by erasing all
square vertices, the identity follows.
\end{proof}
\begin{figure}[htb]
\begin{center}
  \includegraphics[scale=0.6]{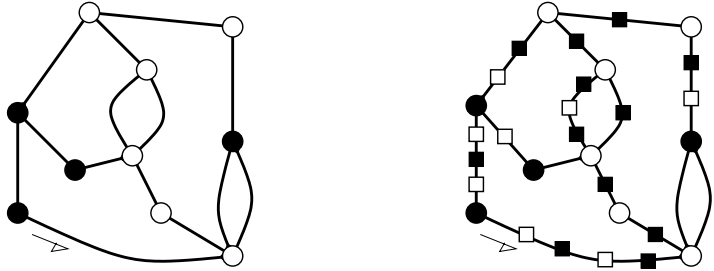}
\end{center}
\caption{A 2-coloured map and one of the associated bipartite maps.}
\label{fig:ising-bip}
\end{figure}

Recall from Sections~\ref{sec:uncoloured-rec}
and~\ref{sec:bij-uncoloured} that the enumeration of 
bipartite maps with prescribed vertex degrees can be addressed via the
recursive method (and equations with one catalytic variable)~\cite{mbm-jehanne}, bijections with blossoming
trees~\cite{mbm-schaeffer-ising}, or bijections with labelled
trees~\cite{bouttier-mobiles}.  In particular, the last two  approaches
explain bijectively\footnote{up to minor restrictions imposed at the
  root of the map} the
algebraicity of the associated \gf, at least when the vertex degrees
are bounded. 

\section{Final comments and questions}
We conclude with a number of questions raised by this survey. The first
type  of question asks what problems have an algebraic solution. Of
course, all methods (recursive, bijective, or via matrix integrals...) are
welcome to answer them. We then go on with a list of problems that have
been solved by a recursive approach, but are still waiting for a purely bijective
proof. We also mention   questions dealing with asymptotic
properties of maps.

\subsection{Algebraicity}
Theorem~\ref{thm:alg-maps} states several algebraicity results for maps with prescribed face
degrees. By comparing the
results dealing with general maps to those dealing with Eulerian
maps, it appears that our understanding of Eulerian maps with
unbounded degrees is probably still incomplete. 
\begin{qn}
Let  $D_\bullet$ and $D_\circ$ be two subsets of $\N$.  Under what
conditions on these sets is the \gf\ of Eulerian maps such that all
black (resp.~white) faces have their degree in  $D_\bullet$ (resp.~$D_\circ$) algebraic?
\end{qn}
This question can in principle be addressed via the equations
of~\cite{mbm-schaeffer-ising,bouttier-mobiles}. Algebraicity is known to hold when  $D_\bullet$ and $D_\circ$ are
finite, and when $D_\bullet=\{m\}$ and $D_\circ=m\N$. A natural
sub-case that could be addressed first is the following\footnote{Having
  raised the question, the author has started to explore it... and
  come with a positive answer~\cite{mbm-m}. Algebraicity can be proved
either via the equations of~\cite{mbm-schaeffer-ising}, or via a bijection with $(m+1)$-constellations.}.
\begin{qn}
 Is the \gf\ of Eulerian planar maps in which all face degrees are
  multiples of $m$ algebraic?
\end{qn}
Recall that for $m\ge 3$, these are the maps that admit a cyclic $m$-colouring
(Section~\ref{sec:more-func-eq}). Algebraicity has already been proved when $m=2$, that
is, for maps that are both Eulerian and
bipartite~\cite{liskovets-walsh,poulalhon-schaeffer-bip-eul}. 


Eulerian maps are required to have even vertex-degrees. But one could
think of other restrictions  than parity.
\begin{qn}
 Under what condition is the \gf\ of maps in which both the vertex
  degrees and the face degrees are constrained algebraic?
\end{qn}
This question seems of course very hard to address. A positive answer
is known in at least one case: the \gf\ of triangulations in which all
vertices have  degree at least $d$ is
algebraic for all $d$~\cite{bernardi-high-v,gao-wormald-cubic}.

\subsection{Bijections}
Our first question may seem surprising at first sight.
\begin{qn}
  Design bijections between families of trees and  families of planar maps with unbounded degrees.
\end{qn}
Indeed, it seems that all bijections that can be used to count
families of maps with unbounded degrees use a detour via maps with
bounded degrees. The simplest example, presented in Section~\ref{sec:bij-uncoloured}, is
that of general planar maps: we have first shown that they are  in bijection with 4-valent maps
(or, dually, quadrangulations), before describing  two types of bijections
between 4-valent maps and trees.
We could actually content ourselves with this situation: after all,
isn't a
combination of two beautiful bijections  twice as beautiful as
a single bijection? But there exist problems with an algebraic solution,
dealing with maps with unbounded degrees, that have not been solved
by a direct bijection so far, like the Ising model on general planar maps
(Theorem~\ref{thm:2}), or the hard-particle model on general planar
maps~\cite{mbm-jehanne}. Discovering such bijections could also give
an algebra-free proof of the fact that maps in which all degrees are
multiples of $m$ are algebraic; the bijection of~\cite{BDG-planaires}
gives indeed a proof, but requires a bit of algebra. Moreover, this
could be a purely bijective way to address the questions raised above on the
algebraicity of Eulerian maps in which all face degrees are multiples
of $m$.
%
%

\begin{qn}
  Design bijections for $q$-coloured maps. 
\end{qn}
This can take several directions: 
\begin{itemize}
\item 
find bijections for the special
values of $q$ (like $q=3$) that are known to yield algebraic \gfs\
(Theorems~\ref{thm:alg-planaires} and~\ref{thm:alg-triang});
\item find bijections for specialisations of the  Potts \gf\
  of maps, like those presented in Section~\ref{sec:bij-coloured} for spanning trees
  and bipolar orientations; 
\item finally, one would dream of designing bijections that would
  establish directly differential equations for coloured maps, starting
  with the (relatively simple?) case of
  triangulations~(see~\eqref{Tutte-ED}).
The author is currently working on  an interesting construction of
Bouttier \emm et al.,~\cite{BDG-blocked}, which  allows to count
  spanning forests on maps and to derive certain differential
  equations in a simpler way than the recursive approach~\cite{mbm-forests}.
 \end{itemize}

Finally, we have discussed in Section~\ref{sec:bij-uncoloured} two families of bijections,
but a third one could exist, as suggested by Bernardi's beautiful
construction for loopless triangulations~\cite{bernardi-kreweras}.
\begin{qn}
  Is the bijection of~\cite{bernardi-kreweras} the tip of some iceberg?
\end{qn}


\subsection{Asymptotics of maps}
\begin{qn}
What is the asymptotic number of properly $q$-coloured planar maps
having $n$ edges?  
\end{qn}
This question has been studied by Odlyzko and
Richmond~\cite{odlyzko-richmond} for triangulations, starting from the
differential equation~\eqref{Tutte-ED}. For $q \in [15/11, 4] \cup
[5,\infty)$, they proved that the number of properly $q$-coloured
triangulations with $n$ faces is of the form $\kappa \mu^n
n^{-5/2}$. The exponent $-5/2$ is typical in the enumeration of
(uncoloured) planar maps. 

The asymptotic behaviour of the number of $n$-edge $q$-coloured
planar maps has been worked out in~\cite{bernardi-mbm} for $q=2$ and $q=3$, using the 
explicit results of Theorems~\ref{thm:2} and~\ref{thm:3}. Again, the exponent is
$-5/2$. The same question can be asked when a parameter $\nu\not = 0$
weights monochromatic edges. For $q=2$, the exponent is still $-5/2$, except at the
critical value $\nu=(3+\sqrt 5)/2$, where it becomes
$-7/3$. See~\cite{Ka86,BK87} for similar results on maps of fixed
vertex degree.

The proofs of these results use the solutions of the difficult functional
equations~\eqref{eq:M} and~\eqref{eq-Tutte}. It would be extremely interesting to
be able to understand the asymptotic behaviour of these numbers (or the singular
behaviour of the associated series) \emm directly from
these equations,. At the moment, we do not known how to do this, even in
the case of one catalytic variable.
\begin{qn}
  Develop a ``singularity analysis''~\cite{flajolet-odlyzko} for equations with catalytic variables.
\end{qn}

Finally, the asymptotic geometry of random uncoloured maps has
attracted a lot of attention in the past few years~\cite{bouttier-guitter-geodesics,bouttier-guitter-three-point,chassaing-schaeffer}, and a limit
object, \emm the Brownian map,, has been identified~\cite{le-gall-maps,marckert-miermont,marckert-mokkadem}. Similar
questions can be addressed for maps equipped with an additional
structure.
\begin{qn}
  Is there a scaling limit for maps equipped with a spanning tree? a
  spanning forest? for properly $q$-coloured maps?
\end{qn}
 The final section
of~\cite{legall-miermont}  suggests a partial, and conjectural answer
to this question for maps equipped with certain statistical physics
models, including the Ising model. 
  Of course, the first point is to determine how the average distance
between two vertices of these maps scales.

\thankyou{The author gratefully acknowledges
the assistance of  Olivier Bernardi in writing this survey.}
\bibliographystyle{plain}
\bibliography{coloured.bib}

\myaddress

\end{document}